\documentclass[11pt]{article}

\usepackage{stmaryrd}
\usepackage{multirow,mathtools,latexsym, enumerate, enumitem, amsmath, amsthm, amssymb,hyperref, pdftricks,tikz,framed,color,enumitem,enumerate,mleftright}
\usetikzlibrary{arrows.meta, decorations.markings,math} 

\usepackage[utf8]{inputenc}   
\usepackage[T1]{fontenc}      
\usepackage{amsmath,amsthm} 
\usepackage{amssymb,mathrsfs} 
\usepackage{amssymb}
\usepackage{amsfonts}
\usepackage{upgreek}
\usepackage{nicefrac}
\usepackage[a4paper]{geometry}
\usepackage{authblk}
\usepackage{graphicx}
\usepackage{hyperref}
\hypersetup{
  colorlinks = true,
  citecolor = blue,
}
 \usepackage{color}
 \usepackage{algorithm2e}
\usepackage{stmaryrd}

\usepackage{enumitem}
\usepackage{url}
\usepackage{graphicx} 
\usepackage{lmodern} 
\usepackage{xcolor}
\usepackage{bbm}

\usepackage{ifthen}
\usepackage{xargs}

\usepackage{todonotes}

\usepackage{aliascnt}
\usepackage{cleveref}
\usepackage{autonum}
\makeatletter
\newtheorem{theorem}{Theorem}
\crefname{theorem}{theorem}{Theorems}
\Crefname{Theorem}{Theorem}{Theorems}

\newtheorem*{lemma_nonumber*}{Lemma}

\newaliascnt{lemma}{theorem}
\newtheorem{lemma}[lemma]{Lemma}
\aliascntresetthe{lemma}
\crefname{lemma}{lemma}{lemmas}
\Crefname{Lemma}{Lemma}{Lemmas}

\newaliascnt{corollary}{theorem}

\aliascntresetthe{corollary}
\crefname{corollary}{corollary}{corollaries}
\Crefname{Corollary}{Corollary}{Corollaries}

\newaliascnt{proposition}{theorem}
\newtheorem{proposition}[proposition]{Proposition}
\aliascntresetthe{proposition}
\crefname{proposition}{proposition}{propositions}
\Crefname{Proposition}{Proposition}{Propositions}

\newaliascnt{definition}{theorem}

\aliascntresetthe{definition}
\crefname{definition}{definition}{definitions}
\Crefname{Definition}{Definition}{Definitions}

\newaliascnt{remark}{theorem}
\newtheorem{remark}[remark]{Remark}
\aliascntresetthe{remark}
\crefname{remark}{remark}{remarks}
\Crefname{Remark}{Remark}{Remarks}

\newtheorem{example}[theorem]{Example}
\crefname{example}{example}{examples}
\Crefname{Example}{Example}{Examples}

\crefname{figure}{figure}{figures}
\Crefname{Figure}{Figure}{Figures}

\newtheorem{assumption}{\textbf{A}\hspace{-3pt}}
\crefformat{assumption}{{\textbf{A}}#2#1#3}

\Crefname{assumptionG}{\textbf{G}\hspace{-3pt}}{\textbf{G}\hspace{-3pt}}
\crefname{assumptionG}{\textbf{G}}{\textbf{G}}

\newtheorem{assumptionH}{\textbf{H}\hspace{-3pt}}
\Crefname{assumptionH}{\textbf{H}\hspace{-3pt}}{\textbf{H}\hspace{-3pt}}
\crefname{assumptionH}{\textbf{H}}{\textbf{H}}

\newtheorem{assumptionE}{\textbf{E}\hspace{-3pt}}
\Crefname{assumptionE}{\textbf{E}\hspace{-3pt}}{\textbf{E}\hspace{-3pt}}
\crefname{assumptionE}{\textbf{E}}{\textbf{E}}

\Crefname{assumptionB}{\textbf{B}\hspace{-3pt}}{\textbf{B}\hspace{-3pt}}
\crefname{assumptionB}{\textbf{B}}{\textbf{B}}

\Crefname{assumptionQ}{\textbf{Q}\hspace{-3pt}}{\textbf{Q}\hspace{-3pt}}
\crefname{assumptionQ}{\textbf{Q}}{\textbf{Q}}

\def\Pens{\mathcal{P}}

\def\msa{\mathsf{A}}

\def\msx{\mathsf{X}}

\def\mcl{\mathcal{L}}

\newcommand{\mcb}[1]{\mathcal{B}(#1)}

\def\mcx{\mathcal{X}}

\def\mcf{\mathcal{F}}

\def\mcp{\mathcal{P}}

\def\rset{\mathbb{R}}

\def\nset{\mathbb{N}}

\def\rmX{\mathrm{X}}
\def\bfrmX{\boldsymbol{\rmX}}
\def\trmX{\tilde{\rmX}}
\def\rmV{\mathrm{V}}
\def\bfrmV{\boldsymbol{\rmV}}
\def\trmV{\tilde{\rmV}}

\newcommand{\rme}{\mathrm{e}}

\newcommand{\rmD}{\mathrm{D}}

\newcommand{\rmR}{\mathrm{R}}
\newcommand{\rmY}{\mathrm{Y}}

\def\rmP{\mathrm{P}}

\def\rmK{\mathrm{K}}
\def\Krm{\mathrm{K}}
\def\rmA{\mathrm{A}}

\def\rmD{\mathrm{D}}

\def\tKrm{{\mathrm{K}}}
\def\rmS{\mathrm{S}}
\def\rmR{\mathrm{R}}

\def\rmQ{\mathrm{Q}}

\def\rmd{\mathrm{d}}

\def\rml{\mathrm{L}}

\def\rme{\mathrm{e}}

\def\rmC{\mathrm{C}}

\def\Ltt{\mathtt{L}}

\newcommand{\tvnorm}[1]{\| #1 \|_{\mathrm{TV}}}
\newcommand{\tvnormEq}[1]{\left \| #1 \right \|_{\mathrm{TV}}}
\newcommandx{\Vnorm}[2][1=V]{\| #2 \|_{#1}}
\newcommandx{\normpi}[2][2=2]{\left\Vert  #1 \right\Vert_{#2}}
\newcommandx{\normH}[2][2=2]{\left\Vert  #1 \right\Vert}
\newcommandx{\normHLigne}[2][2=2]{\Vert  #1 \Vert}
\newcommandx{\normHLine}[2][2=2]{\Vert  #1 \Vert}
\newcommandx{\normmu}[2][2=2]{\left\Vert  #1 \right\Vert_{#2}}
\newcommandx{\normopmu}[2][2=2]{\left\vvvert  #1 \right\vvvert_{#2}}
\newcommandx{\normopH}[2][2=2]{\left\vvvert  #1 \right\vvvert}
\newcommandx{\normop}[2][2=2]{\Vert  #1 \Vert}
\newcommand{\ps}[2]{\left\langle#1,#2 \right\rangle}
\newcommand{\psLigne}[2]{\langle#1,#2 \rangle}

\newcommandx{\normpiLine}[2][2=2]{\Vert  #1 \Vert_{#2}}
\newcommandx{\normmuLine}[2][2=2]{\Vert  #1 \Vert_{#2}}
\newcommandx{\normopmuLine}[2][2=2]{\vvvert  #1 \vvvert_{#2}}
\newcommandx{\normopHLine}[2][2=2]{\vvvert  #1 \vvvert}
\newcommandx{\normopLine}[2][2=2]{\vvvert  #1 \vvvert}

\newcommandx{\VnormEq}[2][1=V]{\left\| #2 \right\|_{#1}}
\newcommandx{\norm}[2][1=]{\ifthenelse{\equal{#1}{}}{\left\Vert #2 \right\Vert}{\left\Vert #2 \right\Vert^{#1}}}
\newcommandx{\normLigne}[2][1=]{\ifthenelse{\equal{#1}{}}{\Vert #2 \Vert}{\Vert #2\Vert^{#1}}}
\newcommandx{\abs}[2][1=]{\ifthenelse{\equal{#1}{}}{\left\vert #2 \right\vert}{\left\vert #2 \right\vert^{#1}}}
\newcommandx{\absLigne}[2][1=]{\ifthenelse{\equal{#1}{}}{\vert #2 \vert}{\vert #2\vert^{#1}}}

\newcommand{\parenthese}[1]{\left(#1 \right)}

\newcommand{\parentheseDeux}[1]{\left[ #1 \right]}
\newcommand{\defEns}[1]{\left\lbrace #1 \right\rbrace }

\newcommandx\probaMarkovTilde[2][2=]
{\ifthenelse{\equal{#2}{}}{{\widetilde{\mathbb{P}}_{#1}}}{\widetilde{\mathbb{P}}_{#1}\left[ #2\right]}}

\newcommand{\PE}{\mathbb{E}}
\newcommand{\expe}[1]{\PE \left[ #1 \right]}

\newcommandx{\expeE}[2][2=]{\mathbb{E}^{#2}\left[ #1 \right]}
\newcommandx{\expeELigne}[2][2=]{\mathbb{E}^{#2}[ #1 ]}
\newcommand{\expeLigne}[1]{\PE [ #1 ]}

\newcommand{\plusinfty}{+\infty}

\def\eqsp{\;}

\newcommand{\coint}[1]{\left[#1\right)}
\newcommand{\ocint}[1]{\left(#1\right]}
\newcommand{\ooint}[1]{\left(#1\right)}
\newcommand{\ccint}[1]{\left[#1\right]}

\def\tv{\mathrm{tv}}
\newcommandx\sequence[3][2=,3=]
{\ifthenelse{\equal{#3}{}}{\ensuremath{\{ #1_{#2}\}}}{\ensuremath{\{ #1_{#2}, \eqsp #2 \in #3 \}}}}
\newcommandx\sequenceD[3][2=,3=]
{\ifthenelse{\equal{#3}{}}{\ensuremath{\{ #1_{#2}\}}}{\ensuremath{( #1)_{ #2 \in #3} }}}

\newcommandx{\sequencen}[2][2=n\in\N]{\ensuremath{\{ #1_n, \eqsp #2 \}}}
\newcommandx\sequenceDouble[4][3=,4=]
{\ifthenelse{\equal{#3}{}}{\ensuremath{\{ (#1_{#3},#2_{#3}) \}}}{\ensuremath{\{  (#1_{#3},#2_{#3}), \eqsp #3 \in #4 \}}}}
\newcommandx{\sequencenDouble}[3][3=n\in\N]{\ensuremath{\{ (#1_{n},#2_{n}), \eqsp #3 \}}}

\def\iid{i.i.d.}

\def\eg{e.g.}

\def\Idd{\operatorname{I}_d}

\def\rmD{\mathrm{D}}

\def\Tr{\mathrm{Tr}}
\def\generator{\mcl}
\def\Langevin{{\bf{L}}}
\def\Hamiltonian{{\bf{H}}}
\def\generatorL{\generator^{\Langevin}}
\def\Discrete{\bf{D}}
\def\generatorD{\generator^{\Discrete}}
\def\generatorH{\generator^{\Hamiltonian}}

\def\bfe{\mathbf{e}}

\def\trace{\operatorname{Tr}}
\def\Tr{\trace}

\def\transpose{\operatorname{T}}

\newcommand\tcr[1]{\textcolor{red}{#1}}

\newcommand{\1}{\mathbbm{1}}

\newcommandx{\CPE}[4][1=,4=]{{\mathbb E}^{#4}_{#1}\left[#2 \, \middle | #3 \right]} 
\newcommandx{\CPELigne}[4][1=,4=]{{\mathbb E}^{#4}_{#1}[#2\, | #3 ]} 
\newcommandx{\CPVar}[3][1=]{\mathrm{Var}^{#3}_{#1}\left\{ #2 \right\}}
\newcommand{\CPP}[3][]
{\ifthenelse{\equal{#1}{}}{{\mathbb P}\left(\left. #2 \, \right| #3 \right)}{{\mathbb P}_{#1}\left(\left. #2 \, \right | #3 \right)}}

\newcommand{\coupling}[1]{\Gamma\left( #1 \right)}

\def\distance{\ell}
\newcommandx{\wasserstein}[3][1=\distance,3=]{\mathbf{W}_{#1}^{#3}\left(#2\right)}
\newcommandx{\wassersteinLigne}[3][1=\distance,3=]{\mathbf{W}_{#1}^{#3}(#2)}
\newcommandx{\wassersteinD}[1][1=\distance]{\mathbf{W}_{#1}}
\newcommandx{\wassersteinDLigne}[1][1=\distance]{\mathbf{W}_{#1}}

\def\tb{\tilde{b}}
\def\tbg{\tilde{b}_{\gamma}}

\def\bgamma{\bar{\gamma}}

\def\dist{\mathbf{d}}
\newcommand\ceil[1]{\lceil #1 \rceil}
\newcommand\floor[1]{\lfloor #1 \rfloor}

\def\bY{\overline{Y}}

\newcommandx{\gperthmc}[2][1=,2=]{\ifthenelse{\equal{#1}{}}{\Xi}{\ifthenelse{\equal{#2}{}}{\Xi_{h,#1}}{\Xi_{#2,#1}}}}
\newcommandx{\Phiverlet}[2][1=,2=]{\ifthenelse{\equal{#1}{}}{\Phi}{\Phi_{#1}^{\circ (#2)}}}
\newcommandx{\gpertub}[2][1=,2=]{\ifthenelse{\equal{#1}{}}{g}{g_{#1}^{#2}}}
\newcommandx{\Phiverletq}[2][1=,2=]{\ifthenelse{\equal{#1}{}}{\widetilde{\Phi}}{\widetilde{\Phi}_{#1}^{\circ (#2)}}}
\newcommandx{\Phiverletqi}[2][1=,2=]{\ifthenelse{\equal{#1}{}}{\bar{\Psi}}{\bar{\Psi}_{#1}^{\circ (#2)}}}
\newcommandx{\Psiverlet}[2][1=,2=]{\ifthenelse{\equal{#1}{}}{\tilde\uppsi}{\tilde\uppsi_{#1}^{\circ (#2)}}}
\newcommandx{\Psiverletq}[2][1=,2=]{\ifthenelse{\equal{#1}{}}{\widetilde{\Psi}}{\widetilde{\Psi}_{#1}^{\circ (#2)}}}
\newcommandx{\Psiverletqi}[2][1=,2=]{\ifthenelse{\equal{#1}{}}{\bar{\Psi}}{\bar{\Psi}_{#1}^{\circ (#2)}}}
\newcommandx{\Pkerhmc}[2][1=,2=]{\ifthenelse{\equal{#1}{}}{\mathrm{P}}{\mathrm{P}_{#1, #2}}}
\newcommandx{\Qkerhmc}[2][1=,2=]{\ifthenelse{\equal{#1}{}}{\mathrm{Q}}{\mathrm{Q}_{#1, #2}}}
\newcommandx{\tPkerhmc}[2][1=,2=]{\ifthenelse{\equal{#1}{}}{\tilde{\mathrm{P}}}{\tilde{\mathrm{P}}_{#1, #2}}}
\newcommandx{\PkerhmcD}[2][1=,2=]{\ifthenelse{\equal{#1}{}}{\mathrm{K}}{\mathrm{K}_{#1, #2}}}

\def\tP{\tilde{P}}

\def\tQ{\tilde{Q}}

\def\pig{\pi_{\gamma}}

\def\tM{\tilde{M}}
\def\varphibf{\boldsymbol{\varphi}}

\def\varphibfd{\boldsymbol{\varphi}_d}

\def\bfGamma{\boldsymbol{\Gamma}}

\def\bfLambda{\boldsymbol{\Lambda}}
\def\Lambdabf{\boldsymbol{\Lambda}}

\def\proj{\operatorname{proj}}

\def\tell{\tilde{\ell}}
\newcommand{\tcrr}[1]{}
\def\eqdef{=}

\def\Jrm{\mathrm{J}}

\def\Hamil{\mathbf{H}}
\def\bfrmX{\mathbf{X}}
\def\bfrmY{\mathbf{Y}}
\def\trmY{\tilde{\mathrm{Y}}}
\def\trmX{\tilde{\mathrm{X}}}

\newcommandx\CTV[1][1=]{\ifthenelse{\equal{#1}{}}{C_{\tv}}{C_{\tv,#1}}}
\newcommandx\ATV[1][1=]{\ifthenelse{\equal{#1}{}}{A_{\tv}}{A_{\tv,#1}}}
\newcommandx\BTV[1][1=]{\ifthenelse{\equal{#1}{}}{B_{\tv}}{B_{\tv,#1}}}

\def\lambdaTV{\lambda_{\tv}}
\def\Phibf{\boldsymbol{\Phi}}
\def\nFun{\mathtt{n}}

\pdfoutput=1

\usepackage{soul}

\title{Asymptotic bias of inexact Markov Chain Monte Carlo methods in high dimension}
\author[1]{Alain Oliviero Durmus}
\author[2]{Andreas Eberle}
\affil[1]{\small{CMAP, CNRS, Ecole Polytechnique, Institut Polytechnique
de Paris, 91120 Palaiseau, France.}}
\affil[2]{\small{Institut f\"{u}r Angewandte Mathematik \\  Universit\"{a}t Bonn \\ Endenicher Allee 60 \\  53115 Bonn, Germany}}

\begin{document}
\footnotetext[1]{Email: alain.durmus@polytechnique.edu}
\footnotetext[2]{Email: eberle@uni-bonn.de}

\maketitle

 \begin{abstract}
{Inexact Markov Chain Monte Carlo methods rely on Markov chains that 
do not exactly preserve the target distribution. Examples include the unadjusted Langevin algorithm (ULA) and unadjusted Hamiltonian Monte Carlo (uHMC).
 This paper establishes bounds on Wasserstein distances between the
invariant probability measures 
of inexact MCMC methods and their target distributions with a focus on understanding the precise dependence of this asymptotic bias on both
dimension and discretization step size.
Assuming Wasserstein bounds on the convergence to equilibrium of 
either the exact or the approximate dynamics, we show that for both
ULA and uHMC, the asymptotic bias depends on key quantities related to the target distribution or the stationary probability measure of the scheme. As a corollary, we conclude that for models with a limited amount of interactions such as mean-field models, finite range graphical models, and
perturbations thereof, the asymptotic bias has a similar 
dependence on the step size and the dimension as for product measures.
}
\end{abstract}

\section{Introduction}

This paper deals with Markov Chain Monte Carlo (MCMC)
methods based on Markov chains that do not exactly preserve a
target distribution $\pi$ on $\rset^d$. A simple example is the Unadjusted Langevin
algorithm (ULA) where the Markov chain is an Euler-Maruyama (EM) discretization of
an overdamped Langevin diffusion with invariant measure
$\pi$. Alternatively, modifications  of the EM scheme can be applied, such as Runge-Kutta and $\theta$ methods \cite{higham:2000,rossler:2010,zygalakis:2011,debussche:faou:2012,abdulle:vilmart:zygalakis:2014} or a
taming strategy \cite{brosse:et:al:spa::2019}. Further, in recent years, variants of the EM scheme have been developed in recent years specifically for certain forms of target distributions motivated by applications in computational statistics and machine learning \cite{welling2011bayesian,dubey2016variance,durmus:moulines:pereyra:2018,durmus2019analysis}.
 Another
important class of inexact MCMC methods is based on Hamiltonian dynamics which are approximated numerically by a symplectic integrator \cite{duane:et:al:1987,neal:1993,leimkuhler:reich:2004}. Unadjusted Hamiltonian Monte Carlo (uHMC) is one of the most popular schemes in this class. It consists of using the Störmer-Verlet (or leapfrog) integrator in combination with momentum randomizations.

In general, the basic idea underlying inexact MCMC methods is to consider a continuous-time Markov process which is ergodic with respect to $\pi$. Since an exact simulation of the continuous time dynamics is
usually not possible, a discrete time approximation is adopted. Although in principle, it is possible in many cases to incorporate a Metropolis-Hastings accept/reject step to ensure that the target distribution $\pi$ is invariant for the corresponding Markov chain, this is not always convenient since it may lead 
to high rejection probabilities and slow convergence of the method,
{ see the discussion in Section \ref{adjustornotadjust}}. On the other hand, without adjustment, the discretization with time step size $\gamma>0$ usually has an invariant measure $\pi_\gamma$
that is only approximately equal to $\pi$, and approaches $\pi$ in the limit $\gamma\to 0$.

{
In recent years, a lot of work has been done on the analysis of the approximation error of inexact MCMC methods, focusing mainly on the unadjusted Langevin algorithm \cite{Dalalyan,durmus:moulines:2016,durmus:moulines:bernoulli, brosse:et:al:spa::2019,debortoli2019convergence,
MajkaMijatovicSzpruch,pages2020unajusted,mangoubi2017rapid,
BouRabeeSchuh,bourabee2021mixing}.
In most of these works, the analysis of approximation bias is intertwined with the study of convergence to equilibrium. While this approach has yielded meaningful results, the intertwining of contraction properties and bias does not make it clear exactly which factors contribute to the dimension dependence of the resulting bounds and in which way. Therefore, to gain a better understanding, we propose to separate the two effects and to divide the error analysis into two parts:}
quantifying the mixing properties of the Markov chain, and quantifying the distance
between its invariant measure $\pi_\gamma$ and the target distribution $\pi$. The focus of this work is on the second
task. In particular, we provide a careful analysis of Wasserstein distances
between $\pi_{\gamma}$ and $\pi$ and their  dependence on both the dimension $d$ and the 
discretization step size $\gamma$. We will see that the results we obtain and their conclusion depend  crucially on
the type of Wasserstein distance that we consider.
\smallskip

There is already an extensive literature on the bias associated with numerical schemes for SDE, and in particular for Euler-Maruyama discretizations. The seminal works \cite{Talay,TalayTubaro} analyze the
difference between the integrals $\int_{\rset^d} f\, \rmd \pi$ and $\int_{\rset^d} f\, \rmd \pi_\gamma$
for smooth functions $f$.
The regularity requirements for $f$ have been relaxed considerably in later work \cite{BallyTalay}.
In \cite{MattinglyStuartTretyakov} the authors bound the distance between $\pi$ and
$\pi_\gamma$ in metrics weaker than Wasserstein distances, and outline possible approaches to deriving Wasserstein bounds. 
Subsequently, convergence to equilibrium and Wasserstein and total variation bias for Euler-Maruyama discretizations have been studied in several papers including
\cite{durmus:moulines:2016,durmus:moulines:bernoulli,MajkaMijatovicSzpruch,pages2020unajusted}.
\tcr{These two distances are of interest for applications in Bayesian inference. The total variation distance, by definition, allows us to obtain guarantees for the estimates for  highest posterior density regions  produced by MCMC algorithms.  As for the Wasserstein distance, it allows for guarantees when one wants to estimate the mean of the posterior distribution.}
\smallskip

Our goal in the present work is to understand more precisely and more generally the order in 
the step size $\gamma$ and the dimension $d$ of Wasserstein distances between $\pi$ and $\pi_\gamma$. To this end, we follow a simple approach
outlined for example in \cite[Remark 6.3]{MattinglyStuartTretyakov}, which is based on a triangle inequality trick, see \Cref{theo:asymptotic_bias_altern} below. To implement this approach,
we need two ingredients: a bound on the convergence to equilibrium in Wasserstein distance for either  the exact dynamics or its numerical approximation, and a bound on the finite time accuracy
of the approximation.  Bounds 
of the first type have been derived systematically in recent years in various
situations \cite{eberle2016, durmus:moulines:2016, durmus:moulines:bernoulli, EberleMajka, MajkaMijatovicSzpruch,debortoli2019convergence,EberleGuillinZimmerAOP,BoEbZi2020,BouRabeeSchuh}. Our main contribution is therefore a careful study 
of the finite time Wasserstein accuracy in high dimension.\smallskip

To explain our main results, we start with a simple but important example 
which can be easily analyzed. Suppose that
$\pi =\mu^d= \bigotimes_{i=1}^d \mu$ is a $d$-fold product of a probability measure $\mu$
on $\mathbb R$ which is absolutely continuous with respect to the Lebesgue measure, with density proportional to $\rme^{-V}$ where $V$ is a continuously differentiable function. Thus, $\pi$ admits a density proportional
to $\rme^{ -U}$ where for any $x=(x_i)_{i=1}^d \in \rset^d$,
\begin{equation}
  \label{eq:U_iid}
  U(x)=\sum_{i=1}^d V(x_i) \eqsp.
\end{equation}
 Under mild assumptions on $V$, the measure
$\pi$ is invariant for the overdamped Langevin diffusion defined by the SDE
$$\rmd Y_t= - \nabla U (Y_t)\, \rmd t\, +\, \sqrt 2\, \rmd B_t \eqsp,$$
where $(B_t)_{t \geq 0}$ is a Brownian motion in $\mathbb R^d$.
Now consider the Euler discretization
\begin{equation}
  \label{eq:def_euler_marua}
  X_{k+1}= X_k - \gamma \nabla U(X_k) + \sqrt{2\gamma}\, G_{k+1} \eqsp,
\end{equation}
where  $\gamma >0$ is the step size, and $(G_k)_{k \geq 1}$ is a sequence of independent
standard normal random variables in $\mathbb R^d$. The recursion \eqref{eq:def_euler_marua}
defines a Markov chain with state space $\mathbb R^d$ and transition kernel
$\rmR_{\gamma}$. 
The \emph{unadjusted Langevin algorithm} (ULA) consists in simulating the Markov chain $(X_k)_{k\geq 0}$
 to get approximate samples from $\pi$.\smallskip

{ Recall that for a metric 
$\dist :\mathbb R^d\times \mathbb R^d\to \rset_+$, $p\in [1,\plusinfty )$,
and probability measures 
$\mu ,\nu\in\mcp(\mathbb R^d)$, the $\rml^p$ Wasserstein distance of order $p $ associated with $\dist$ is defined by
\begin{equation}
  \wasserstein[p,\dist]{\mu,\nu} = \inf\defEns{\int_{\mathbb R^d\times\mathbb R^d} \dist (x,y)^p\,  \rmd \zeta(x,y) \, : \, \zeta \in \coupling{\mu,\nu} }^{1/p} \eqsp,
\end{equation}
where
$\coupling{\mu,\nu}$ is the set of all couplings of $\mu$ and $\nu$,
i.e., all probability measures on $\mathbb R^d\times\mathbb R^d$ with
marginals $\mu$ and $\nu$.
The $\rml^p$ Wasserstein distance associated with the
Euclidean distance is denoted by $\mathbf W_p$.}
Under mild assumptions, it can be shown that the Markov chain $(X_k)_{k \geq 0}$ defined by \eqref{eq:def_euler_marua} has a unique invariant probability 
measure $\pi_\gamma$, and the $\rml^p$ Wasserstein distance
$\wassersteinD[p](\pi_\gamma ,\pi)$ is of order $O(\gamma )$ for every $p\in [1,2]$. 
This is well-known \cite{durmus:moulines:bernoulli}, and follows also from the results below.
We are interested in the precise dependence of the corresponding 
bounds on both the dimension and the step size. In the particular case where $U$ is of the form \eqref{eq:U_iid}, the analysis is relatively simple. Indeed, it is easy to verify that under mild assumptions on $V$,
$\pi_\gamma$ is the $d$-fold product of the invariant measure $\mu_\gamma$ corresponding to ULA with one-dimensional target distribution $\mu$, and, therefore,
$$\wassersteinD[2](\pi_\gamma ,\pi)^2= \wassersteinD[2](\mu^d_\gamma ,\mu^d)^2= d\, \wassersteinD[2](\mu_\gamma ,\mu )^2\eqsp.$$
Thus for $p=2$, $\wassersteinD[p](\pi_\gamma ,\pi)$ is of order $O(d^{1/2}\gamma )$, and the same 
holds for any $p\le 2$, since in this case, $\wassersteinD[p]\le \wassersteinD[2]$.
More generally, for any $q \geq 1$, we can endow $\mathbb R^d$ with  either the $\ell^q$ distance
\begin{equation}
  \label{eq:def_ell_q}
\textstyle {\ell^q}(x,y)= \left\| x-y\right\|_{\ell^q}= \left(\sum_{i=1}^d|x_i-y_i|^q\right)^{1/q}\eqsp,
\qquad x,y\in\mathbb R^d\eqsp,
\end{equation}
or the normalized $\tell^q$ distance
\begin{equation}
  \label{eq:def_tell_q}
  \textstyle {\tell^q}(x,y)= \left\| x-y\right\|_{\tell^q}= \left( d^{-1}\sum_{i=1}^d|x_i-y_i|^q\right)^{1/q} \eqsp,
\qquad x,y\in\mathbb R^d\eqsp,
\end{equation}
and consider the corresponding Wasserstein distances $\wassersteinD[p,\ell^q]$ and $\wassersteinD[p,\tell^q]$ of order $p$
on the space of probability measures on $\mathbb R^d$.
Note that for any $q\in [1,2]$, we have $\| x\|_{\ell^q}= d^{1/q}\| x\|_{\tell^q}$ and $\| x\|_{\tell^q}\le \| x\|_{\tell^2}$. Therefore,
for any $p,q\in [1,2]$,
\begin{eqnarray*}
\wassersteinD[p,\tell^q](\pi_\gamma ,\pi)& \le & \wassersteinD[2,\tell^2](\pi_\gamma ,\pi)\ 
= d^{- 1/2}\,\wassersteinD[2](\pi_\gamma ,\pi) \  \in\ O(\gamma )\eqsp,\\
\wassersteinD[p,\ell^q](\pi_\gamma ,\pi)& =& d^{ 1/q}\, \wassersteinD[p,\tell^q](\pi_\gamma ,\pi)
\ \in\ O(d^{1/q}\gamma ) \eqsp. 
\end{eqnarray*}
On the other hand, an explicit computation in the case where $\mu$ and $\mu_\gamma$ are
Gaussian measures shows that, at least for $q=2$, this order is sharp, see Section \ref{sec:gauss} below. Thus in the product case, to obtain an accurate approximation of the invariant 
measure w.r.t.\ the $\wassersteinD[p,\ell^q]$ distance,
the step size $\gamma$ in the unadjusted Euler scheme should be chosen of order
$O(d^{-1/q})$, whereas an accurate approximation in the $\wassersteinD[p,\tell^q]$
distance can be achieved with a step size that is independent of the dimension.
It follows that if one is only interested in 
approximating integrals $\int_{\rset^d} f\, \rmd\pi$ for
functions $f:\mathbb R^d\to\mathbb R$ that are Lipschitz continuous w.r.t.\ the $\tell^q$ metric with a Lipschitz constant that
does not depend on the dimension $d$, then the step size can be chosen \emph{independently of $d$}. This is often the case in molecular dynamics simulations when $f$ is an 
intensive quantity. 
Examples include averages $f(x)=\frac 1d\sum_{i=1}^d\Phi (x_i)$, and more generally, U-statistics 
$f(x)={d\choose k }^{-1}\sum_{1\le i_1<i_2<\ldots <i_d}^d\Phi (x_{i_1},\ldots ,x_{i_k} )$, where $k\in \{ 1,\ldots ,d\}$ is fixed, and $\Phi  :\rset^k\to\rset $ is Lipschitz continuous.
If, on the other hand, one is interested in the integrals of functions
that are Lipschitz continuous w.r.t.\ the $\ell^q$ metric with a fixed dimension free Lipschitz constant,
then a step size of order $O(d^{-1/q})$ is required. This scenario is more common in 
applications in Bayesian statistics and machine learning \cite{robert:2007,barber:2012}. 
\smallskip

An alternative to ULA is the \emph{unadjusted Hamiltonian Monte Carlo} algorithm (uHMC) \cite{duane:et:al:1987,neal2011mcmc,BoSaActaN2018,BoEbZi2020,DurmusMoulinesSaksman} which is based on the Hamiltonian flow $(\uppsi_t)_{t \geq 0}$ associated to the 
unit mass Hamiltonian $H(q,p)=U(q)+|p|^2/2$, 
i.e., $\uppsi_T(q_0,p_0) = (q_T,p_T)$ where $(q_t,p_t)_{t \geq 0}$ is the solution of the ordinary differential equation 
$\frac{\rmd}{\rmd t} (q_t,p_t)  =  (p_t,-\nabla U(q_t))$ with initial value $(q_0,p_0)$. Fix $T >0$, let $(G_k)_{k \geq 1}$ be a sequence of independent standard normal random variables, and denote by $\proj_q : \rset^{2d} \to \rset$ the projection onto the first $d$ components. Then the recursion $  Q_{k+1} = \proj_q(\uppsi_T(Q_k,G_{k+1}))$
 defines a Markov chain for which $\pi$ is invariant. This Markov chain 
 corresponds to the \emph{exact Hamiltonian Monte Carlo (xHMC)} algorithm.
To be able to carry out numerical 
computations, the Hamiltonian flow is approximated using the Verlet scheme with a given time step size $\gamma >0$, or an alternative integrator, see 
Section \ref{sec:uHMC}. The MCMC method using the Markov chain defined as above, 
but with the exact Hamiltonian flow $(\uppsi_t)_{t \geq 0}$ replaced by its numerical approximation $(\tilde\uppsi_t)_{t \geq 0}$, i.e., $  \tilde{Q}_{k+1} = \proj_q(\tilde{\uppsi}_T(  \tilde{Q}_k,G_{k+1}))$,
is referred to as the \emph{unadjusted Hamiltonian Monte Carlo (uHMC)} algorithm.
It can be shown under mild assumptions that the corresponding transition kernel $  \Krm_{T,\gamma}$
has an invariant probability measure $\pi_{T,\gamma}$ such that 
$\wassersteinD[p](\pi_{T,\gamma },\pi)$ is of order $O(\gamma^2 )$ for any $p\in [1,2]$. The improved order compared to ULA comes from
the fact that the Verlet scheme is a higher order integrator. Once more,
we are interested in the precise dependence of the corresponding 
bounds on the dimension and the step size. 
In the case where $\pi$ is a product measure associated with $U$ of the form \eqref{eq:U_iid}, $\pi_{T,\gamma }$ is also a $d$-fold product of the
invariant measure $\mu_{T,\gamma}$ associated with uHMC with target distribution $\mu$. Thus
following similar arguments as for ULA, we obtain
\begin{align}
\wassersteinD[2](\pi_{T,\gamma },\pi)& = \wassersteinD[2](\mu^d_{T,\gamma} ,\mu^d)= d^{1/2}\, \wassersteinD[2](\mu_{T,\gamma },\mu )\ \in\ O(d^{1/2}\gamma^2)\eqsp,\\
\wassersteinD[p,\tell^q](\pi_{T,\gamma } ,\pi)& \le  \wassersteinD[2,\tell^2](\pi_{T,\gamma } ,\pi)\ 
= d^{- 1/2}\,\wassersteinD[2](\pi_{T,\gamma } ,\pi) \  \in\ O(\gamma^2 ) \eqsp,\\
\wassersteinD[p,\ell^q](\pi_{T,\gamma } ,\pi)& = d^{ 1/q}\, \wassersteinD[p,\tell^q](\pi_{T,\gamma } ,\pi)
\ \in\ O(d^{1/q}\gamma^2 )\eqsp. 
\end{align}
Again, these bounds are sharp if $\mu$ is a Gaussian measure and $q=2$, see Section 
\ref{sec:gauss}.
Thus in the product case, the situation is completely analogous for uHMC
as for ULA, except that the dependence of the orders on $\gamma $ is better for uHMC. In particular,
for an accurate approximation of the invariant 
measure w.r.t.\ the $\wassersteinD[p,\ell^q]$ distance,
the discretization step size $\gamma$ in uHMC should be chosen of order
$O(d^{-1/(2q)})$, whereas an accurate approximation in 
$\wassersteinD[p,\tell^q]$ can be achieved again with a step size that is
independent of the dimension.
\smallskip

Our goal in this paper is to study under which assumptions
results similar to the ones described above hold. For ULA as well as for unadjusted Hamiltonian Monte Carlo, we will see that in the general case where $\pi$ admits a smooth density proportional to $\rme^{-U}$ with respect to the Lebesgue measure, the dimension dependence enters in an 
explicit way through
some key quantities depending on $\nabla U$. In particular, $|\Delta\nabla U|^2$ turns out to be crucial for controlling the dimension dependence -- \tcr{see the discussion after \Cref{theo:bias_final_ula_1} and Section \ref{example:EulerhighD}}.  As a consequence, we can show that for a broad class of models, the
dimension dependence is under appropriate assumptions \emph{of the same order as in the product case}. 
Besides product models,
this class of \emph{``nice'' models} includes finite range graphical models,
mean-field models, and their perturbations (e.g.,
finite dimensional projections of 
measures on infinite dimensional spaces  
that are absolutely continuous w.r.t.\ a Gaussian reference measure), see  \Cref{example:EulerhighD}. In particular, to the authors' knowledge, the class of models that we identify seems to include essentially all models
for which scaling limits of Metropolis-Hastings algorithms have been established; see for example \cite{RobertsGelmanGilks,RobertsRosenthal,YangRobertsRosenthal,
PillaiStuartThiery,BeskosPillaiRobertsSanzSernaStuart}. On the other hand, there is a more general class of models
for which our bounds have a worse dimension dependence as in the product
case. We expect that this is not a coincidence but that the dimension 
dependence of the asymptotic bias may be  generically worse.

Our main results are stated in Section \ref{sec:mainresults}. In Section
\ref{sec:appl-exampl}, we study the resulting dimension dependence for 
concrete classes of models, {and we compare what is known for unadjusted and for Metropolis-adjusted methods}. Most of the proofs of our results are gathered in \Cref{sec:proofs}.

\section*{Notation}
If $\msx$ is a topological space then we denote by $\mathcal B(\msx )$ the
corresponding Borel $\sigma$-field, and by $\mcp(\msx)$ the set
of probability measures on $(\msx,\mcb{\msx})$. 
The Euclidean norm and the Euclidean inner product on $\rset^d$ are denoted by $\absLigne{\cdot}$ and $\psLigne{\cdot}{\cdot}$ respectively, 
and we set
$\mcp_{p}(\rset^d) = \{ \mu \in \mcp(\rset^d) \, : \, \int_{\rset^d}
\abs[p]{x} \mu(\rmd x) < \plusinfty \}$.  
We denote by
$\rmC^k(\rset^d, \rset^m)$ the set of $k$-times continuously
differentiable functions from $\rset^d$ to $\rset^m$, and
$\rmC^k(\rset^d)$ stands for $\rmC^k(\rset^d, \rset)$. For
$f :\rset^d \to \rset^m$, denote by
$\nabla f : \rset^d \to \rset^{d \times m}$ the gradient of $f$ and $\Delta f$ the vector Laplacian of $f$ if they exist.
For any function $f:\rset^d\to \rset^m$, $\partial_if$ denotes the partial derivative with respect to the $i$-th variable of $f$ and $\rmD f$ is the differential of $f$.  $\mathrm{div}$ stands for the divergence operator defined by $\mathrm{div}(\psi) = \sum_{i=1}^d \partial_i \psi_i$, where $\psi_i$ is the $i$-th component of $\psi$. 
$\ceil{\cdot}$ and $\floor{\cdot}$ stands for the upper and lower 
integer part, respectively.
For any matrix $A = (A_{i,j})_{i,j=1}^d \in\rset^{d\times d}$, $\mathrm{Tr}(A) = \sum_{i=1}^d A_{i,i}$ denotes the trace of $A$.
Finally, we denote by  $\varphibfd ( x)= (2\uppi)^{-d/2}\exp(-\abs{x}^2/2)$ the density of the $d$-dimensional standard normal distribution.

\section{Main results }\label{sec:mainresults}

Before specializing to more specific settings, 
we start with some simple but important general observations that are the basis for all the results below. Let  $(\msx,\mcx)$ be a measurable space, and suppose that
$ \mathbf W :\mcp(\msx)\times \mcp(\msx) \to [0, \plusinfty]$ is a distance function on the space $\mcp(\msx)$ consisting of all probability measures on $(\msx,\mcx)$.
Note that we allow the value infinity for the distance. 
The bounds on distances between invariant measures that we derive below are all based
on the following lemma.  
  

\begin{lemma}[The triangle inequality trick]
  \label{theo:asymptotic_bias_altern}
Let $\rmQ$  and $\rmS$ be Markov transition kernels
on  $(\msx,\mcx)$ with invariant probability measures  $\pi_{\rmQ}$ and $\pi_{\rmS}$,  respectively.
Suppose that there exist functions
$\varphi ,\varepsilon : \nset \to \rset_+$ with
$\inf_{n \in \nset} \varphi(n) < 1$ such that for any $n \in \nset$,
\begin{eqnarray}
  \label{eq:varphi_bound}
  \mathbf W({\pi_{\rmS}\rmQ^n , \pi_{\rmQ}}) &\leq & \varphi (n) \mathbf W({\pi_{\rmS}, \pi_{\rmQ}}) \eqsp, \qquad \text{and}\\
  \mathbf W({\pi_{\rmS}\rmS^n , \pi_{\rmS}\rmQ^n}) &\leq & \varepsilon (n) \eqsp. \label{eq:epsilon_accuracy}
\end{eqnarray}
Then,
  \begin{eqnarray}\label{eq:accuracyinvariantmeasure2}
   \mathbf W({\pi_{\rmS},\pi_{\rmQ}})& \leq & 
  \inf\left\{ \frac{\varepsilon (n)}{1-\varphi (n)}: n\in\nset\text{ with }\varphi (n)<1\right\} .
  \end{eqnarray}
\end{lemma}

\begin{proof}
By the triangle inequality and the invariance of $\pi_{\rmQ}$ and $\pi_{\rmS}$ w.r.t.\ $\rmQ$ and $\rmS$, we get that for any $n \in \nset$, 
\begin{equation}
  \label{eq:asymptotic_bias_0_1_altern}
     \mathbf W({\pi_{\rmQ},\pi_{\rmS}} )\leq      \mathbf W({\pi_{\rmQ},\pi_{\rmS} \rmQ^n}) +  \mathbf W({\pi_{\rmS} \rmQ^n, \pi_{\rmS}}) \le 
     \varphi (n) \mathbf W({\pi_{\rmQ}  ,\pi_{\rmS}  } )+  \varepsilon (n)\eqsp. 
  \end{equation}
The conclusion follows by rearranging and minimizing over $n$.
\end{proof}

Based on \Cref{theo:asymptotic_bias_altern}, 
if we have a bound $\varphi (n)$ quantifying the convergence to equilibrium for the Markov chain
with transition kernel $\rmQ$,
then we can derive upper bounds on the distance $    \mathbf W(\pi_{\rmS},\pi_{\rmQ})$ by controlling the accuracy $\varepsilon (n)$ for the approximation of the stationary Markov chain with initial distribution
$\pi_{\rmS}$ and transition kernel $\rmS$ by 
 the Markov chain with the same initial distribution
and transition kernel $\rmQ$.
%
This approach is not new and appears in variations at several places in the literature, see for example \cite[Remark 6.3]{MattinglyStuartTretyakov} and \cite{JohndrowMattingly}.
Of course, it can also be applied with the r\^oles of $\rmQ$ and $\rmS$ interchanged, which yields different bounds.\smallskip


\begin{example}\label{example:triangle} Suppose that there exist  $A,B,c,\lambda , \gamma\in (0,\plusinfty )$ such that Conditions \eqref{eq:varphi_bound} and \eqref{eq:epsilon_accuracy} are satisfied with $$\varphi (n)=A\exp (-cn\gamma )\quad\text{ and }\quad\varepsilon (n)
  =\gamma B\exp (\lambda n\gamma )\eqsp.$$ Then by choosing
  $n=\left\lceil (c\gamma)^{-1}\{\log (A)+\log (1+c/\lambda
    )\}\right\rceil$, we obtain the upper bound
  \begin{equation}\label{eq:accuracyinvariantmeasure2}
   \mathbf W({\pi_{\rmS},\pi_{\rmQ}})\leq \ 
  \frac{\gamma B \exp (\lambda n\gamma )}{1-A\exp (-cn\gamma )}\ \le 
\gamma B   \rme^{1+\lambda\gamma }A^{\lambda /c}\left( \frac{\lambda}{c}+1\right) \eqsp.
  \end{equation}
  In the applications we are interested in, typically $\gamma$ is a small constant (the discretization step size), and $c<\lambda$.
Note that $A=1$ can be guaranteed by 
choosing the distance $\mathbf W$ in an adequate way, see the examples in \Cref{sec:appl-exampl}.
\end{example}

More generally, we can also apply \Cref{theo:asymptotic_bias_altern} if the distance to equilibrium of the Markov chain with transition kernel $\rmQ$ 
decays subgeometrically:

\begin{example}\label{example:triangle2} Suppose that there exist  $B,\lambda , \gamma\in (0,\plusinfty )$ and a decreasing continuous function $\psi :\mathbb R_+\to\mathbb R_+$ with $\lim_{t\to\plusinfty}\psi (t)=0$ such that \eqref{eq:varphi_bound} and \eqref{eq:epsilon_accuracy} are satisfied with 
$$\varphi (n)=\psi (n\gamma )\quad\text{ and }\quad\varepsilon (n)
=\gamma B\exp (\lambda n\gamma )\eqsp.$$ 
Let $t_{\mathrm{rel}}\eqdef\inf\{ t\ge 0:\psi (t)\le 1/2\}$.
Then, choosing $n\eqdef\left\lceil t_{\mathrm{rel}}/\gamma \right\rceil$, we obtain 
  \begin{equation}\label{eq:accuracyinvariantmeasure3}
   \mathbf W({\pi_{\rmS},\pi_{\rmQ}})\leq \ 
  \frac{\gamma B\exp (\lambda n\gamma )}{1-\psi (n\gamma )}\ \le 
  2 \gamma B\exp\left({\lambda\cdot (t_{\mathrm{rel}}+\gamma)}\right) \eqsp.
  \end{equation}
If $\psi$ is decaying exponentially then this bound is weaker than the one in \eqref{eq:accuracyinvariantmeasure2}.
\end{example}

\smallskip

In this work, our focus is on quantifying the dependence on the dimension of
corresponding bounds for Markov processes on $\mathbb R^d$.
As distance functions on $\mathcal P(\mathbb R^d)$ we consider
$\rml^p$ Wasserstein distances $\mathbf W=\wassersteinD[p,\dist]$
where $p\in [1,2]$ and $\dist : \mathbb R^d\times \mathbb R^d \to \rset_+$ is a lower semicontinuous distance function on $\mathbb R^d$.
It is important to note that there is some flexibility in choosing
the underlying metric $\dist$. It is this flexibility that will often
enable us to satisfy the conditions in Example \ref{example:triangle}
with $A=1$.\smallskip

We assume that $\dist$ is upper bounded by the
Euclidean distance:
\begin{assumption}
    \label{assum:distance} 
    There exists  $C_\dist\in (0,\plusinfty ) $ such that for any
    $x,y\in\mathbb R^d$,
    $$\dist (x,y)\ \le C_\dist |x-y|. $$
 \end{assumption}   
For example, if $\dist $ is the $\ell^q$ distance defined in  \eqref{eq:def_ell_q} for some $q\in [1,2]$, then \Cref{assum:distance} holds with $C_\dist =d^{\frac{1}{q}-\frac{1}{2}}$, and  if $\dist $ is the $\tilde\ell^q$ distance defined in \eqref{eq:def_ell_q}, then $C_\dist=1$.\medskip

Obtaining precise information 
on the dimension dependence of 
$ \wasserstein[p,\dist]{\pi_{\rmS},\pi_{\rmQ}}$
using \Cref{theo:asymptotic_bias_altern} requires 
bounds as stated in \eqref{eq:varphi_bound} and \eqref{eq:epsilon_accuracy} with explicit dimension dependence of $\varphi (n)$
and $\varepsilon (n)$. 
Regarding the former, in recent years, dimension free contractions in appropriate Wasserstein distances have been
proven under different assumptions for various important classes of Markov processes
including overdamped Langevin diffusions and more general Kolmogorov processes
\cite{eberle2016,EberleGuillinZimmerTAMS}, corresponding Euler discretizations \cite{EberleMajka,debortoli2019convergence}, second order Langevin diffusions
\cite{cheng2017underdamped}, and both exact and unadjusted Hamiltonian Monte Carlo
\cite{mangoubi2017rapid,BoEbZi2020,BouRabeeSchuh}. It is well-known
that such contractions immediately imply upper bounds as assumed in 
\eqref{eq:varphi_bound} and \eqref{eq:epsilon_accuracy}. For the reader's convenience, a short proof of this fact is included in
\Cref{sec:proof-crefl}. 

Besides convergence bounds for the reference kernel, 
the second key ingredient for studying the dimension dependence of the distance
between two invariant measures is an accuracy bound as in \eqref{eq:epsilon_accuracy} that quantifies
the distance between the
laws at time $n$ of the corresponding Markov chains started with the same initial 
distribution. Such bounds depend on the approximation that is considered and can only be derived on a case-by-case basis. 
The precise dimension dependence of the function $\varepsilon (n)$ in these bounds in different situations is one of the main contributions of this work.

\subsection{Euler-Maruyama discretizations of stochastic differential equations}
\label{sec:asympt-bias-euler}

Consider a diffusion process $(Y_t)_{t \geq 0}$ on $\mathbb R^d$ that solves a stochastic differential equation (SDE) 
\begin{equation}
  \label{eq:sde}
  \rmd Y_t = b(Y_t)\,  \rmd t + \sqrt 2\,\rmd B_t \eqsp,
\end{equation}
where $(B_t)_{t \geq 0}$ is a $d$-dimensional Brownian motion and $b: \rset^d \to \rset^d$ is a twice continuously differentiable function. We assume that \eqref{eq:sde} admits a unique non-explosive solution $(Y_t)_{t \geq 0}$ for every starting point $x \in \rset^d$. Moreover, we impose the following assumption on the Markov semigroup
$(\rm\rmP_t)_{t \geq 0}$ defined by $(Y_t)_{t \geq 0}$.
\begin{assumptionE}
    \label{assum:diffusion_semigroup_a} $(\rm\rmP_t)_{t \geq 0}$ admits an invariant probability measure $\pi \in \mcp_2(\rset^d)$.
\end{assumptionE}
In particular, the assumption is satisfied if $\pi$ admits a density with respect to the Lebesgue measure of the form
\begin{equation}
  \label{eq:density_pi}
  \pi(\rmd x) = \mathrm{Z}^{-1}\, \rme^{-U(x)} \, \rmd x \eqsp, \qquad \mathrm{Z} = \int \rme^{-U(x)} \rmd x < \plusinfty \eqsp,
\end{equation}
for a function $U\in \rmC^3(\mathbb R^d)$ satisfying
$\int_{\rset^d} (1+|x|^2)\rme^{-U(x)}\, \rmd x <\plusinfty$, and if $b=-\nabla U+\Xi$ for a $\rmC^2$ vector field $\Xi :\mathbb R^d\to\mathbb R^d$ such that $\mathrm{div}(\rme^{-U}\Xi )=0$
(e.g.\ $\Xi =-\Jrm\nabla U$ for an antisymmetric matrix $\Jrm\in\rset^{d\times d}$), see \cite{eberle:lecture:notes:markov}.\smallskip

We consider Euler-Maruyama type discretization schemes for \eqref{eq:sde}, i.e., the class of Markov chains $(X_k)_{k \ge 0}$ defined by the following recursion: for any integer $k \ge 0$,
\begin{equation}
  \label{eq:def_euler_maru}
  X_{k+1} = X_k + \gamma\, \tbg(X_k) + \sqrt{2\gamma}\, G_{k+1} \eqsp,
\end{equation}
where  $\gamma >0$ is the step size, $(G_k)_{k \in\nset}$ is a sequence of independent zero-mean  Gaussian
random variables on $ \rset^d$ with covariance matrix identity, and $\{\tbg : \rset^d \to \rset^d \, : \, \gamma \in\ocint{0,\bgamma} \}$, with $\bgamma >0$, is a
family of approximate drift functions satisfying the following condition.
\begin{assumptionE}
  \label{as:tilde_b_b} There exists a function $\bfGamma : \rset^d \to \rset_+$  such that for any $\gamma >0$ and $x \in \rset^d$,
  \begin{equation}
 \abs{\tbg(x) - b(x)}   \leq \gamma \bfGamma(x) \eqsp.
  \end{equation}
\end{assumptionE}
For the standard Euler-Maruyama scheme, $\tbg = b$ for any $\gamma \in \ocint{0,\bgamma}$, and therefore \Cref{as:tilde_b_b} is satisfied with $\bfGamma = 0$. In the case $b=-\nabla U$, the Euler scheme corresponds to the standard Unadjusted Langevin Algorithm (ULA) \cite{roberts:tweedie:1996}, but as mentioned previously, $b=-(\Idd+\Jrm)\nabla U$ with an antisymmetric matrix $\Jrm\in\rset^{d\times d}$ is also an option to target  $\pi$ of the form \eqref{eq:density_pi}.
Moreover, our conditions also cover the tamed Euler-Maruyama discretization {\cite{brosse:et:al:spa::2019}} for which $\tbg (x)= b(x)/(1+\gamma |{b(x)}|)$. In this case, \Cref{as:tilde_b_b} holds with $\bfGamma (x)= |b(x)|$.\smallskip

The transition kernel of the Markov chain defined by the recursion \eqref{eq:def_euler_maru}
is
\begin{equation}
  \label{eq:def_rmR_gamma}
  \rmR_{\gamma}(x,\msa) \  \tcrr{= \mathcal N(x+\gamma \tbg(x),2\gamma \Idd)[A] \, }= (2\gamma )^{-d/2} \int_{\msa} \varphibfd\parenthese{\frac{y-x-\gamma \tbg(x)}{\sqrt{2\gamma}}} \rmd y \eqsp.
\end{equation}
We assume the following condition on the family $\{ \rmR_{\gamma}  : \gamma \in\ocint{0,\bgamma} \}$.
\begin{assumptionE}
  \label{as:r_gamma}
For every $\gamma \in \ocint{0,\bgamma}$, $\rmR_{\gamma}$ has an invariant probability measure $\pi_{\gamma} \in \mcp_2(\rset^d)$.
\end{assumptionE}

We aim at applying \Cref{theo:asymptotic_bias_altern}
in order to obtain explicit upper bounds on Wasserstein distances of the 
invariant measures
$\pi$ and $\pi_{\gamma}$ for
$\gamma \in \ocint{0,\bgamma}$.  
This can be achieved by choosing either  
$\rmQ = \rmR_{\gamma}$
and $\rmS = \rmP_{\gamma}$ in \Cref{theo:asymptotic_bias_altern}, or, conversely, $\rmQ = \rmP_{\gamma}$
and $\rmS = \rmR_{\gamma}$. 
Both approaches
lead to slightly different results
{that are \emph{not comparable} to each other
, see Theorems \ref{theo:bias_final_ula_1} and 
\ref{theo:bias_final_ula_2} below, respectively.
In particular, one either requires a convergence bound on the approximate
dynamics as assumed in  \eqref{eq:convergence_ula_bound_W_1}, or
a convergence bound on the exact dynamics as assumed in  \eqref{eq:convergence_ula_bound_W_2}.
}

\subsubsection{A first result}
The first main result stated in Theorem \ref{theo:bias_final_ula_1} has a simple form and is relatively
easy to derive but requires stronger assumptions. In particular,
we assume a global Lipschitz condition on the approximate
drift functions $\tbg$. 
\begin{assumptionE}
  \label{as:b_lip}
  There exists  $L \in \mathbb R_+ $ such that for any $x,y \in \rset^d$ and $\gamma \in \ocint{0,\bgamma}$, 
  $$\absLigne{\tbg(x) - \tbg(y)} \leq L \absLigne{x-y} .$$ 
\end{assumptionE}
We consider the extended generator of \eqref{eq:sde} given  for $f \in \rmC^2(\rset^d)$ by
\begin{equation}\label{eq:defgenL}
  \generator^{\Langevin} f\  = \ps{b}{\nabla f} + \Delta f \eqsp. 
\end{equation}
 For any twice continuously differentiable function $F : \rset^d \to \rset^n$, we define $\generatorL F$ component-wise, i.e., $\generatorL F$ is the function from $\rset^d $ to $\rset^n$ with 
$i$-th component given by $\generatorL F_i$ where $F_i$ is the $i$-th component of $F$. Let $\| A\|_{\mathrm F}$ denote the Frobenius (or Hilbert-Schmidt) norm 
of a matrix $A \in\rset^{d \times d}$, i.e.,
$\| A\|_{\mathrm F}^2=\sum_{i,j=1}^dA_{i,j}^2$.
\begin{proposition}
  \label{theo:bound_w_2_euler}
  Assume \Cref{assum:diffusion_semigroup_a}, \Cref{as:tilde_b_b}
  and \Cref{as:b_lip}.
  Then for any $\gamma \in \ocint{0,\bgamma}$ and $n \in \nset$,
\begin{equation}
\label{eq:lem:bound_w_2_euler_res}
\mathbf W_2(\pi \rmP_{n \gamma} , \pi \rmR_{\gamma}^n)\   \leq   \gamma  M_{\Langevin}^{1/2} \rme^{\lambda_{\Langevin} n \gamma}  \eqsp,
\end{equation}
where 
 \begin{align}
   \label{eq:def_lambda_euler}
   \lambda_{\Langevin} & =1+L^2+3 L^2\bgamma /2\, , \\
 M_{\Langevin}  &= \left(6^ {-1}+3\gamma/4 \right) M_1 + 3 M_2/2 + (1+3\gamma/2) M_3 \eqsp,
\end{align}
with
\begin{equation}
  \label{eq:def_m_1_m_2_m_3}
M_1 = \int{|\generatorL b|^2}\, \rmd\pi  \eqsp, \quad 
M_2 = \int{\norm[2]{\rmD b}_{\mathrm F}}\, \rmd\pi \eqsp, \quad M_3 = \int{\bfGamma}^2\, {\rmd\pi} \eqsp.
\end{equation}
\end{proposition}

The proof of the proposition is given in \Cref{sec:proof-ula_One}. Of course, it
is well-known that the Euler-Maruyama approximation is accurate of order $O(\gamma )$, see for example \cite{Talay}. The point of Theorem \ref{theo:bound_w_2_euler} is however that the 
explicit form of the prefactor $M_{\Langevin}^{1/2}$ enables us to analyze
precisely the dimension dependence for different classes of models,
see Section \ref{example:EulerhighD} below. 

{Recall the definition of the constant $C_\dist  $ from Assumption  \Cref{assum:distance} above.}
By combining 
\Cref{theo:bound_w_2_euler}, \Cref{theo:asymptotic_bias_altern}, and Examples \ref{example:triangle} and
\ref{example:triangle2} we obtain our first main result.

\begin{theorem}
  \label{theo:bias_final_ula_1}
  Assume \Cref{assum:diffusion_semigroup_a}, \Cref{as:tilde_b_b}, \Cref{as:r_gamma} and \Cref{as:b_lip}, and fix $p\in [1,2]$.
Suppose that $\dist$ is a distance function on $\mathbb R^d$ 
satisfying \Cref{assum:distance}, and assume that there exist 
$A\geq 0$ and $c>0$ such that for
any $\gamma \in \ocint{0,\bgamma}$ and $n\in \mathbb N$,
  \begin{equation}
    \label{eq:convergence_ula_bound_W_1}
    \wassersteinLigne[p,\dist]{\pi_{\gamma} , \pi \rmR_{\gamma}^{n}}\leq A \rme^{-c n \gamma}\wassersteinLigne[p,\dist]{\pi_{\gamma} , \pi} \eqsp.
  \end{equation}
Let $\lambda_{\Langevin}$ and $M_{\Langevin}$ be defined as in  \Cref{theo:bound_w_2_euler}.  Then for any $\gamma \in \ocint{0,\bgamma}$,
  \begin{equation}
     \label{eq:bias_final_ula_1a}
\wassersteinLigne[p,\dist]{\pi_{\gamma} , \pi} \leq \     \gamma C_\dist   M_{\Langevin}^{\frac12}\,  \rme^{1+\lambda_{\Langevin}\gamma }A^{\lambda_{\Langevin} /c}\left(\lambda_{\Langevin}/c+1\right)  \eqsp.
\end{equation}
More generally, suppose instead of \eqref{eq:convergence_ula_bound_W_1}
that there exists a decreasing continuous function $\psi :\mathbb R_+\to\mathbb R_+$ with $\lim_{t\to\plusinfty}\psi (t)=0$ such that  for
any $\gamma \in \ocint{0,\bgamma}$ and $n\in \mathbb N$,
\begin{equation}
    \label{eq:convergence_ula_bound_W_1b}
    \wassersteinLigne[p,\dist]{\pi_{\gamma} , \pi \rmR_{\gamma}^{n}}\leq \psi { (n \gamma )}\wassersteinLigne[p,\dist]{\pi_{\gamma} , \pi} \eqsp.
  \end{equation}
  Let $t_{\mathrm{rel}}\eqdef\inf\{ t\ge 0:\psi (t)\le 1/2\}$.
Then for any $\gamma \in \ocint{0,\bgamma}$,
 \begin{equation}
     \label{eq:bias_final_ula_1b}
\wassersteinLigne[p,\dist]{\pi_{\gamma} , \pi} \leq \     2\gamma\,  C_\dist   M_{\Langevin}^{\frac12} \,  \rme^{\lambda_{\Langevin}\cdot(t_{\mathrm{rel}}+\gamma) }  \eqsp. 
\end{equation}
   
\end{theorem}

Theorem \ref{theo:bias_final_ula_1}  is a direct consequence of   \Cref{theo:bound_w_2_euler} and the bounds in \eqref{eq:accuracyinvariantmeasure2} and \eqref{eq:accuracyinvariantmeasure3}.
In particular, if $L$, $A$, $c$, $\lambda_{\Langevin}$ and $\psi$ are independent of the dimension $d$, then 
the dimension dependence of the upper bounds is determined 
completely by the key quantities $C_\dist$ and $M_\Langevin$.
{In Section \ref{example:EulerhighD}, we will see that for ULA,
the dimension dependence of $M_\Langevin$ relies crucially on bounds for 
$\left|\Delta b(x)\right|^2$ and $\left\| \rmD^2 b(x)\right\|^2_{\mathrm{F}}$
where $b=-\nabla U$, cf.~\eqref{eqLaplace} and \eqref{eqFronbenius2D}.
If these quantities are bounded uniformly of order $O(d)$
then under appropriate assumptions, 
the resulting dimension dependence 
on upper bounds of the standard $\rml^1$ Wasserstein distance $\mathbf W_1(\pi_\gamma ,\pi )$
is of order $O(\gamma d^{1/2})$. 
This is the case for models with a limited amount of interactions such as
product models, mean-field models, finite range graphical models, and
perturbations thereof. On the other hand, for general models where the second partial derivatives of $b$ are bounded, one can only
expect bounds on the above quantities of order $O(d)$ and hence bounds
on the asymptotic bias of order $O(\gamma d)$.}

\subsubsection{An improved result}
In our second main result stated in Theorem \ref{theo:bias_final_ula_2} below, we relax
the assumptions substantially, see the comments below Theorem \ref{theo:bias_final_ula_2}. In contrast to Theorem \ref{theo:bias_final_ula_1},
we only assume a quantitative convergence bound for the diffusion process, and, more importantly,
we replace the global Lipschitz condition in \Cref{as:b_lip}  by the following one-sided Lipschitz condition on $b$.
\begin{assumptionE}
  \label{ass:one_side_lip_b}
  There exists $\kappa >0$ such that for any $x,y \in \rset^d$,
  \begin{equation}
    \label{eq:222}
    \ps{b(x) - b(y)}{x-y} \leq \kappa \norm[2]{x-y} \eqsp.
  \end{equation}
\end{assumptionE}

For any $f \in \rmC^2(\rset^d)$, we define $\generatorD f : \rset^d \times \rset^d \to \rset$,  for $x,y \in \rset^d$ by
\begin{equation}
\label{eq:def_generator_discrete}
  (\generatorD f)(x,y) =   \ps{b(x)}{(\nabla f)(y) } + (\Delta f)(y) \eqsp. 
\end{equation}
For any twice continuously differentiable function $F : \rset^d \to \rset^n$, we define $\generatorD F$ component-wise, i.e., $\generatorD F$ is the function from $\rset^d\times \rset^d $ to $\rset^n$ with 
$i$-th component given by $\generatorD F_i$ where $F_i$ is the $i$-th component of $F$.

\begin{proposition}
  \label{theo:bound_w_2_euler_altern}
    Assume \Cref{assum:diffusion_semigroup_a}, \Cref{as:tilde_b_b}, \Cref{as:r_gamma} and \Cref{ass:one_side_lip_b}.
    Then for any $\gamma \in \ocint{0,\bgamma}$ and $n \in \nset$,
\begin{equation}
\label{eq:lem:bound_w_2_euler_res_altern}
\wasserstein[2]{\pi_{\gamma} \rmP_{n\gamma}, \pi_{\gamma}\rmR_{\gamma}^n} \leq  \gamma^{3/2}(n +1)^{1/2}\tM_{\Langevin}^{1/2}\rme^{1/2+\kappa n\gamma } \leq  \gamma(1+\gamma)^{1/2} \tilde M_{\Langevin}^{1/2} \,  \rme^{(1+\kappa)n\gamma}  \eqsp,
\end{equation}
where
\begin{equation}
  \label{eq:tM_lang}
    \tM_{\Langevin}=\frac 16  \tM_1    +\frac 12\gamma^{1/2}\tM_2^{1/2}\tM_3^{1/2} + \frac{1}{\sqrt 2}\tM_2^{1/2}\tM_4^{1/2}+ \frac 12\tM_5 \eqsp,
\end{equation}
with
\begin{align}
\tM_1       &  =    \sup_{u \in \ocint{0, \gamma}} \int_{\rset^d}\int_{ \rset^d}  \abs[2]{  (\generatorD b)(x,x+u \tbg (x) + (2u)^{1/2}z) }  \varphibf_d(z)\, \rmd z\; \pig  (\rmd x)  \eqsp,
     \\
\tM_2      & =  \sup_{u \in \ocint{0, \gamma}} \int_{\rset^d }\int_{ \rset^d}   \norm[2]{(\rmD b) (x+u \tbg (x) + (2u)^{1/2}z)}_{\mathrm F}  \, \varphibf_d(z)\,\rmd z\; \pig(\rmd x)     \eqsp,\\
\tM_3      & = \sup_{t \in \rset_+}\int |\generatorL b|^2\, \rmd (\pig \rm\rmP_t)  \eqsp, \quad 
 \tM_4 = \sup_{t \in \rset_+}  \int{\norm[2]{\rmD \text{$b$} }_{\mathrm F}}\, \rmd (\pig \rm\rmP_t)  \eqsp,\\
\tM_5 & = \int \bfGamma^2\, \rmd \pig \eqsp.
\end{align}
\end{proposition}

The proof of the proposition is given in \Cref{sec:proofs-crefth} below. By combining 
\Cref{theo:bound_w_2_euler_altern}, \Cref{theo:asymptotic_bias_altern}, and Examples \ref{example:triangle} and
\ref{example:triangle2}, we obtain our second main result.

\begin{theorem}
  \label{theo:bias_final_ula_2}
  Assume \Cref{assum:diffusion_semigroup_a}, \Cref{as:tilde_b_b}, \Cref{as:r_gamma} and \Cref{ass:one_side_lip_b}, and fix $p\in [1,2]$.
Suppose that $\dist$ is a distance function on $\mathbb R^d$ 
satisfying \Cref{assum:distance}, and assume that there exist 
 $A\geq 0$ and $c>0$ such that for any 
$t\geq 0$ and any probability measure $\nu\in\mathcal P(\mathbb R^d)$,
  \begin{equation}
    \label{eq:convergence_ula_bound_W_2}
    \wassersteinLigne[p,\dist]{ \nu \rmP_{t},\pi }\leq A \rme^{-c t}\wassersteinLigne[p,\dist]{\nu , \pi} \eqsp.
  \end{equation}
Let $\tM_{\Langevin}$ be defined as in \Cref{theo:bound_w_2_euler_altern}.  Then for any $\gamma \in \ocint{0,\bgamma}$,
  \begin{equation}
     \label{eq:bias_final_ula_2a}
\wassersteinLigne[p,\dist]{\pi_{\gamma} , \pi} \leq \     \gamma\, C_\dist   \tM_{\Langevin}^{\frac12}  \rme^{1+(1+\kappa)\gamma }A^{(1+\kappa )/c}\left(1+(1+\kappa)/c\right) \eqsp .
\end{equation}
More generally, suppose instead of \eqref{eq:convergence_ula_bound_W_1}
that there exists a decreasing continuous function $\psi :\mathbb R_+\to\mathbb R_+$ with $\lim_{t\to\plusinfty}\psi (t)=0$ such that for any 
$t\geq 0$ and any probability measure $\nu\in\mathcal P(\mathbb R^d)$,
\begin{equation}
    \label{eq:convergence_ula_bound_W_2b}
    \wassersteinLigne[p,\dist]{\nu \rmP_{t},\pi  }\leq \psi { (t )}\wassersteinLigne[p,\dist]{\nu , \pi} \eqsp.
  \end{equation}
  Let $t_{\mathrm{rel}}\eqdef\inf\{ t\ge 0:\psi (t)\le 1/2\}$.
Then for any $\gamma \in \ocint{0,\bgamma}$,
 \begin{equation}
     \label{eq:bias_final_ula_2a}
\wassersteinLigne[p,\dist]{\pi_{\gamma} , \pi} \leq \     2\gamma  C_\dist   \tM_{\Langevin}^{\frac12} \,  \rme^{(1+\kappa )\cdot(t_{\mathrm{rel}}+\gamma) }\eqsp.
\end{equation}
\end{theorem}

A main feature of \Cref{theo:bias_final_ula_2} is that the bounds depend 
on the one-sided Lipschitz constant $\kappa$ of the unperturbed drift $b$ and on the convergence to equilibrium of the diffusion process, whereas the bounds
in \Cref{theo:bias_final_ula_1} depend on the
global Lipschitz constant $L$ of the perturbed drift $\tilde b_\gamma$ and on
the convergence to equilibrium of the approximating process.
In particular, if $\kappa$, $A$, $c$ and $\psi$ are independent of the dimension $d$, then the dimension
dependence of the upper bounds is determined completely by the quantities
$C_\dist$ and $\tM_i$, $i=1,2,\ldots ,5$. The price 
to pay is that these quantities take a more complicated form than the
corresponding quantities $M_1$, $M_2$ and $M_3$ occurring in \Cref{theo:bias_final_ula_1}. It turns out that nevertheless, $\tM_{\Langevin}$
can be bounded in applications similarly as $M_{\Langevin}$, see the examples in Section \ref{example:EulerhighD}.

\tcr{We conclude this section by noting that \Cref{theo:bias_final_ula_2} easily implies convergence guarantees for ULA, if it is combined with either convergence bounds in Wasserstein distance for the Langevin diffusion or its discretization as established in \cite{debortoli2019convergence,EberleMajka,eberle:2015,EberleGuillinZimmerTAMS}. To illustrate our point, assume under the conditions of \Cref{theo:bias_final_ula_2} that \eqref{eq:convergence_ula_bound_W_2} holds and $A,c,\kappa,C_\dist$ do not depend on the dimension. Suppose, moreover, that the Wasserstein 
distance $\wassersteinLigne[p,\dist]{\mu_0 , \pi}$ between
the initial distribution $\mu_0$ and the target distribution $\pi$ is of order $O(d^{\varpi})$ with $\varpi >0$.  Then, by \Cref{theo:bias_final_ula_2} and the triangle inequality, the number of steps $n_{\varepsilon}\in\mathbb N$ sufficient to achieve
$\wassersteinLigne[p,\dist]{\mu_0 \rmR_{\gamma}^{n_{\varepsilon}}, \pi} \leq \varepsilon$ for an expected precision $\varepsilon >0$,
is of order $O(\varepsilon^{-1} \tM_{\Langevin}^{1/2})$ (up to logarithmic terms), and thus it is of order $O(\varepsilon^{-1} d)$ if $b$ has bounded second derivative (see \Cref{sec:appl-exampl}). 
}

\subsubsection{Asymptotic bias in total variation}
\label{sec:asympt-bias-total}

By combining the results of the previous sections with those in \cite{debortoli2019convergence},
we can also provide explicit bounds on $\tvnorm{\pi_{\gamma} - \pi}$. We consider the  following general conditions.
\begin{assumptionE}
  \label{ass:wasser_to_tv}
  \begin{enumerate}[wide, labelwidth=!, labelindent=0pt, label=(\roman*)]
  \item  \label{ass:wasser_to_tv_i} There exist $\lambdaTV>0$ and  $\ATV \geq 0$, such that for any  $\gamma \in \ocint{0,\bgamma}$ and $n \in \nset$,
    \begin{equation}
      \label{eq:2}
      \wassersteinLigne[1]{\pi_{\gamma} \rmP_{n\gamma}, \pi_{\gamma} \rmR_{\gamma}^n} \leq \ATV \gamma (\gamma n)^{1/2}\exp\parenthese{\gamma\, n\, \lambdaTV}  \eqsp.
    \end{equation}
  \item  \label{ass:wasser_to_tv_ii} There exists  $\BTV \geq 0$, such that for any $\gamma \in \ocint{0,\bgamma}$, $      \wassersteinLigne[1]{\pi,\pi_{\gamma}} \leq \BTV \gamma$. 
  \item  \label{ass:wasser_to_tv_iii} There exists  $\CTV \geq 0$ such that for any $t \in \ccint{0,2}$ and $x,y \in\rset^d$,
    \begin{equation}
      \label{eq:16}
      \tvnorm{\updelta_x \rmP_t - \updelta_y \rmP_t} \leq \CTV \norm{x-y}t^{-1/2} \eqsp.
    \end{equation}
  \end{enumerate}
\end{assumptionE}
Note that under appropriate assumptions, we can show that \Cref{ass:wasser_to_tv}-\ref{ass:wasser_to_tv_i}-\ref{ass:wasser_to_tv_ii} holds by applying \Cref{theo:bound_w_2_euler_altern} and \Cref{theo:bias_final_ula_2}, respectively. In particular, the expression of $\lambdaTV$ provided by these results does not depend explicitly on the dimension $d$. Moreover, the results established in \cite{debortoli2019convergence} allow us to verify the bound in
\Cref{ass:wasser_to_tv}-\ref{ass:wasser_to_tv_iii} with some explicit constants. For example, we can show the following statement for which the proof is postponed to \Cref{sec:proofs-crefth-crefth_tv}.
\begin{theorem}
  \label{theo:regular_transf}
  Assume \Cref{as:b_lip}-\Cref{ass:one_side_lip_b} and $\sup_{x \in\rset^d} \ps{b(x)}{x} < \plusinfty$. Then \Cref{ass:wasser_to_tv}-\ref{ass:wasser_to_tv_iii} holds with  $\CTV =  \sqrt{\kappa/\uppi} \sup_{u \in [0,2]}  \{u/(1-\rme^{-2 \kappa u})^{1/2}\}$. 
\end{theorem}
        Note that the expression for  $\CTV$ provided by \Cref{theo:regular_transf} does not depend on the dimension $d$. 
It would be possible to relax the global Lipschitz assumption on $b$ to a local Lipschitz condition, but this would require to introduce many additional technical details in the resulting proof.

We now state the main result of this section. The proof is postponed to \Cref{sec:proofs-crefth-crefth_tv} and is adapted from the proof of \cite[Corollary 12]{durmus:moulines:bernoulli} which considers the case $b = -\nabla U$ with a strongly convex function $U$. 

\begin{theorem}
  \label{theo:bound_tv}
  Assume \Cref{assum:diffusion_semigroup_a}, \Cref{as:r_gamma} and \Cref{ass:wasser_to_tv}. Suppose in addition that $\bgamma <1$. Then, for any $\gamma \in \ocint{0,\bgamma}$,
  \begin{equation}
    \label{eq:19}
    \tvnorm{\pi-\pi_{\gamma}} \leq  2^{-3/2} L  \gamma \defEns{d+\gamma \tilde{M}_6/3 }^{1/2}+ \gamma \CTV \BTV  + \gamma \ceil{\log(\gamma^{-1})/\log(2)}\tilde{M_7} \eqsp,
  \end{equation}
  where
  \begin{equation}
    \label{eq:def_tilde_M_6}
\tilde{M}_6 =     \int_{\rset^d} \norm[2]{b} \rmd \pi_{\gamma}\eqsp, \quad \tilde{M}_7 = 4  \CTV \ATV  \exp(2 \lambdaTV)\eqsp.
\end{equation}
\end{theorem}

\tcr{
We end this section with the same remark as in the Wasserstein distance case regarding convergence guarantees for ULA implied by \Cref{theo:bound_tv}. By the same reasoning, this result combined with convergence bounds for the Langevin diffusion or its discretization as established in \cite{debortoli2019convergence,EberleMajka} easily leads to complexity bounds for ULA in the total variation distance.
}





\subsection{Unadjusted Hamiltonian Monte Carlo}\label{sec:uHMC}

In this section, we are interested in establishing non-asymptotic bounds between the invariant distributions of the \emph{exact Hamiltonian Monte Carlo algorithm (xHMC)} and the \emph{unadjusted Hamiltonian Monte Carlo algorithm (uHMC)}.  Let $b : \rset^d \to \rset^d$ be a twice continuously differentiable and Lipschitz continuous function, and fix $T\in (0,\plusinfty)$.
We consider a Markov chain $(Q_k,P_k)_{k \ge 0}$ with state space $\mathbb R^d \times \mathbb R^d$ defined recursively by
\begin{equation}
(Q_{k+1},P_{k+1}) = \uppsi_T(Q_k,G_{k+1}) \eqsp,
\end{equation}
where $(G_k)_{k \in\nset}$ is a sequence of \iid~$d$-dimensional zero-mean Gaussian random variables with covariance matrix identity, and  
$(\uppsi_t)_{t \geq 0}$ is the differentiable flow associated to the ordinary differential equation
\begin{equation}
  \label{eq:def_hamil_ode}
  \frac{\rmd}{\rmd t} (q_t,p_t) \  =  (p_t,b(q_t)) \eqsp,
\end{equation}
i.e., $\uppsi_T(q_0,p_0) = (q_T,p_T)$ where $(q_t,p_t)_{t \geq 0}$ is
the solution of \eqref{eq:def_hamil_ode} with initial value
$(q_0,p_0)$. In particular, in the case $b = -\nabla U$,
$(\uppsi_t)_{t \geq 0}$ is the Hamiltonian flow associated to the unit
mass Hamiltonian $$H(q,p)=U(q)+|p|^2/2\eqsp,$$ and correspondingly,
$(Q_k,P_k)_{k \ge 0}$ is the Markov chain associated to the xHMC
algorithm with complete momentum refreshment.    The sequence
$(Q_k)_{k \ge 0}$ is a Markov chain with state space $\mathbb R^d$ and
transition kernel
\begin{align}
  \label{eq:def_Krm}
  \Krm_{T}(q,\msa) &= \int_{\rset^d} \1_{\msa\times\mathbb R^d }\left(\uppsi_T(q,p_0)\right)\, \varphibf_{d}(p_0)\, \rmd p_0    \eqsp, \qquad q \in \rset^d,\quad
  \msa \in \mcb{\rset^d} \eqsp .
\end{align}
Note that similarly to \Cref{sec:asympt-bias-euler},
we consider a general setup where $b$ is a vector field which is not
assumed to be the gradient of a real-valued function. In addition, we assume the existence of a stationary distribution. 

\begin{assumptionH}
  \label{assum:exact_hmc_kernel}
The Markov kernel $ \Krm_{T}$ admits an invariant probability measure $\pi \in \mcp_2(\rset^d)$.
\end{assumptionH}
In the case where $b = -\nabla U$ for some potential function 
$U : \rset^d \to \rset$ such that $\int_{\rset^d} (1+|q|^2)\rme^{-U(q)} \rmd q < \plusinfty$ 
\Cref{assum:exact_hmc_kernel} is always satisfied with  $\pi$ of the form \eqref{eq:density_pi},
see \eg~\cite{BoSaActaN2018,neal2011mcmc}.
In practice, \eqref{eq:def_hamil_ode} can usually not be solved exactly,  and therefore numerical schemes are used to get
approximate solutions. In this paper, we consider discretization with the leapfrog (or Verlet)
integrator. More generally, and analogously as above for the
Euler-Maruyama discretization, we consider a family 
$\{\tbg : \rset^d \to \rset^d \, : \, \gamma \in\ocint{0,\bgamma} \}$, $\bgamma >0$, of approximate drift functions satisfying the following condition.
\begin{assumptionH}
  \label{ass:app_F_h}
  There exists $\Lambdabf : \rset^d \to \rset_+$ such that for any $\gamma >0$ and $q \in \rset^d$,
  \begin{equation}
    \abs{b(q) - \tbg(q)} \leq \gamma^2 \Lambdabf(q) \eqsp. 
  \end{equation}
\end{assumptionH}
Then, the corresponding uHMC algorithm with discretization step size
$\gamma >0$ satisfying $T/\gamma\in\mathbb N$ is given by the Markov chain $(\tQ_{k}, \tP_{k})_{k\ge 0}$ with state space $\rset^d\times\rset^d$
that is defined recursively by
\begin{equation}
(\tQ_{k+1}, \tP_{k+1}) = \Psiverlet_T(\tQ_k,G_{k+1})  \eqsp,
\end{equation}
with  $\Psiverlet_t (q_0,p_0)=(q_t,p_t)$,  where $(q_t,p_t)_{t\ge 0}$
is the (unique) solution of 
\begin{equation}
  \label{eq:def_verlet_ode}
  \frac{\rmd}{\rmd t} (q_t,p_t) \
  =  \left(p_{\lfloor t/\gamma\rfloor\gamma} 
  -(\gamma /2)\tbg (q_{\lfloor t/\gamma\rfloor\gamma})\, ,\, 
  (1/2)\{\tbg(q_{\lfloor t/\gamma\rfloor\gamma})+\tbg(q_{\lceil t/\gamma\rceil\gamma})\}\right) 
\end{equation}
with initial value $(q_0,p_0)$. 
In particular, for any $n\in\mathbb N$, $\Psiverlet_{n\gamma}
=\Psiverlet[\gamma][n]$ where $\Psiverlet_\gamma : \rset^{2d} \to \rset^{2d}$ is given by
\begin{equation}
\label{eq:defPhiverlt_h}
\Psiverlet_\gamma  (q_0,p_0) = (q_\gamma ,p_\gamma )  \eqsp, \qquad 
\begin{cases}
  q_\gamma &= q_0 -(\gamma^2/2) \tbg(q_0) + \gamma p_0\eqsp, \\
  p_\gamma &= p_0 -(\gamma/2)\{\tbg(q_0) + \tbg (q_\gamma )\} \eqsp.
\end{cases}
\end{equation}
If $\tbg = b$ then the function $\Psiverlet_\gamma$ corresponds to one step of the leapfrog (or Verlet) integrator. Again, 
$(\tQ_k)_{k \ge 0}$ is a  Markov chain, and the transition kernel is 
\begin{align}
   \tKrm_{T,\gamma }(q,\msa) &=  \int_{\rset^d} \1_{\msa\times\mathbb R^d}(\Psiverlet_T(q,p_0)) \, \varphibf_{d}(p_0)\, \rmd p_0  \eqsp, \qquad q \in \rset^d,\quad
  \msa \in \mcb{\rset^d} \eqsp.
\end{align}
Similarly as for xHMC, we also assume that the uHMC chain has an invariant probability measure if the step size $\gamma$ is small enough.

\begin{assumptionH}
  \label{as:K_N_T}
  For every $\gamma\in [0,\bar\gamma ]$,
the Markov kernel $\tKrm_{T,\gamma}$  has an invariant probability measure ${\pi}_{T,\gamma}$.
\end{assumptionH}
Finally, we assume that $b$ satisfies a global Lipschitz condition.
\begin{assumptionH}
  \label{as:F_lip}
  There exists $\Ltt \geq 0$ such that for any $q_1,q_2 \in \rset^d$, 
  $$\abs{b(q_1) - b(q_2)} \leq \Ltt \abs{q_1-q_2}.$$ 
\end{assumptionH}

\subsubsection{Main results for unadjusted HMC}\label{subsec:HMC}

The \emph{Liouville operator}, i.e., the infinitesimal generator of the deterministic dynamics  \eqref{eq:def_hamil_ode}, is given for any $f \in \rmC^1(\rset^{2d})$ and $(q,p) \in \rset^{2d}$  by
\begin{equation}
\label{eq:def_generatorH}
  \generatorH f(q,p) =  \ps{p}{\nabla_q f(q,p)} + \ps{b(q)}{\nabla_p f(q,p)} \eqsp .
\end{equation}
For a continuously differentiable function  $F : \rset^d \to \rset^n$, we define $\generatorH F$ component-wise as the function from $\rset^d \to \rset^n$ such that the $i$-th component is $\generatorH F_i$ where $F_i$ is the $i$-th component of $F$.
Note that if $(q_t,p_t)_{t \geq 0}$ is a solution of \eqref{eq:def_hamil_ode}, then $t \mapsto F(q_t,p_t)$ is continuously differentiable on $\rset_+$ and
\begin{equation}
  \label{eq:use_generatorH}
\rmd F(q_t,p_t)/\rmd t =   (\generatorH F)(q_t,p_t) \eqsp. 
\end{equation}
Remarkably, the Liouville operator is related to the generator $\generator^\Langevin$ defined in  \eqref{eq:defgenL}. In particular,
applying $\generatorH$ twice to the function $(q,p)\mapsto b(q)$ yields for $(q,p) \in\rset^{2d}$, $(\generatorH  b) (q,p)=(p^{\transpose}\nabla b)(q)$ and
$
((\generatorH )^2 b) (q,p)= \sum_{i=1}^d \rmD^2b_i(q) \{p \otimes p\} \bfe_i +(b^{\transpose}\nabla b)(q) $, where $(\bfe_i)_{i=1}^d$ is the canonical basis of $\rset^d$ and $b_i$ is the $i$-th component of $b$.
Thus a short computation shows that for any $q\in\mathbb R^d$,
\begin{align}\label{eq:Liouvillenormsquared}
\int \left| (\generatorH  b)(q,p)\right|^2 \,  \, \varphibf_{d}(p)\, \rmd p
&=  \left\| \nabla  b(q)\right\|^2_{\mathrm F} \eqsp, \\
\int \left((\generatorH )^2 b\right) (q,p)\,  \, \varphibf_{d}(p)\, \rmd p
&=  \left(\generator^\Langevin b\right)(q)\eqsp, \label{eq:Liouvillesquared}\\
\int \left|\left((\generatorH )^2 b\right) (q,p)\right|^2\,  \, \varphibf_{d}(p)\, \rmd p
&= \left| \left(\generator^\Langevin b\right)(q)\right|^2 +2 \left\| (\rmD^2 b)(q)\right\|^2_{\mathrm F} \eqsp,\label{eq:Liouvillesquaredsquared}
\end{align}
where we set $\left\| (\rmD^2 b)(q)\right\|^2_{\mathrm F}= \sum_{i=1}^d \left\| (\rmD^2 b_i)(q)\right\|^2_{\mathrm F}=\sum_{k,j,i=1}^d(\partial_{kj}b_i(q))^2$ and used that for any matrix $A=(A_{i,j})_{i,j=1}^d \in \rset^{d\times d}$, denoting by $ \mathbf{1}_{d \times d}$ the $d\times d$ matrix with all entries equal to $1$,
\begin{align}
  \int \{p^{\transpose} A p\}^2 \varphibf_d(p) \rmd p &=   \int \mathrm{Tr}(A^{\transpose} p^{\transpose} p p^{\transpose} p A) \varphibf_d(p) \rmd p\\
  & = \mathrm{Tr}(A^{\transpose}\{2\Idd + \mathbf{1}_{d \times d}\} A) =2\norm{A}_{\mathrm{F}} + \mathrm{Tr}(A)^2 \eqsp.
\end{align}
It is a consequence of these identities that the same constants $M_1$ and $M_2$ as introduced 
in \eqref{eq:def_m_1_m_2_m_3} above,
are also relevant to quantify the accuracy of uHMC.

\begin{proposition}[One-step accuracy of uHMC]
  \label{theo:wasserstein_2_bound}
Assume  \Cref{assum:exact_hmc_kernel}, \Cref{ass:app_F_h}
 and \Cref{as:F_lip}.
Then for any  
$T >0$ and $\gamma\in (0,\bar\gamma ]$
with $T/\gamma\in\mathbb N$, we have 
  \begin{equation}
    \wasserstein[2]{{\pi}\rmK_{T }    ,{\pi} \rmK_{T,\gamma}}\   \leq
    \gamma^2  \, \Ltt^{-1}
\rme^{ \lambda_{\Hamil} T}  M_{\Hamiltonian}^{1/2}   \eqsp ,\text{ where } \lambda_{\Hamil} = \Ltt^{1/2}\left( 1+\frac{\gamma \Ltt^{1/2}}{2} +\frac{\gamma^2 \Ltt}{4} \right) ,
  \end{equation}
and 
$$  
M_{\Hamiltonian} = M_1+(1+\gamma^2\Ltt)\Ltt M_2+2M_4+((2+\gamma\Ltt/2)^2+\Ltt )M_5 \eqsp,
$$
with 
\begin{eqnarray}
  \label{eq:def_m_1_m_2_m_4}
M_1& =& \int{|\generatorL b|^2}\, \rmd\pi  , \quad
M_2 = \int{\norm[2]{\rmD b}_{\mathrm F}}\, \rmd\pi , \quad M_4= \int \left\| \rmD^2 b\right\|^2_{\mathrm F}\, {\rmd\pi}  ,\\
 \label{eq:def_m_5}
M_5 & =& \sup_{t \in [0,T]} \int_{\rset^d}\int_{ \rset^d} \bfLambda \left(\proj_{q}(\Psiverlet_t(q,p))\right)^2\,   \varphibf_d(p)\, \rmd p\; \pi  (\rmd q) 
 \eqsp .
\end{eqnarray}
Here $\proj_{q} : \rset^d \times \rset^d \to \rset^d$ denotes the projection onto the first $d$ components.
\end{proposition}

The proof of the theorem is given in \Cref{sec:proof-crefth}. It can also be extended easily to quantify the accuracy for multiple transition steps, but for 
the applications below, it turns out to be sufficient to consider only a single
transition step of uHMC (which usually already involves many Verlet steps).
{Recall the definition of the constant $C_\dist  $ from Assumption  \Cref{assum:distance} above.}

\begin{theorem}
  \label{theo:wasserstein_2_bound_final_hmc}
  Assume  \Cref{assum:exact_hmc_kernel}, \Cref{ass:app_F_h},  \Cref{as:K_N_T} and \Cref{as:F_lip}, and fix $T\in (0,\plusinfty )$ and $p\in [1,2]$.  
Suppose that $\dist$ is a distance function on $\mathbb R^d$ 
satisfying \Cref{assum:distance}, and assume that there exists
a constant $c\in (0,1]$ such that for
any $\gamma \in \ocint{0,\bgamma}$,
 \begin{equation}
    \label{eq:convergence_uHMC_bound_W_1}
    \wassersteinLigne[p,\dist]{\pi_{T,\gamma} \rmK_{T,\gamma} , \pi  \rmK_{T,\gamma}}\leq (1-c)\wassersteinLigne[p,\dist]{\pi_{T,\gamma} , \pi} \eqsp.
  \end{equation}
  Let $\lambda_{\Hamil}$ and $M_{\Hamiltonian}$ be defined as in
  \Cref{theo:wasserstein_2_bound}.
Then for any  $\gamma\in (0,\bgamma ]$
with $T/\gamma\in\mathbb N$, 
  \begin{equation}
     \label{eq:bias_final_uHMC_1a}
\wassersteinLigne[p,\dist]{\pi_{T,\gamma } , \pi} \leq \    {\gamma^2}  \,(c\Ltt)^{-1}
\rme^{ \lambda_{\Hamil} T} C_\dist   M_{\Hamiltonian}^{1/2} .
\end{equation} 
\end{theorem}

\begin{proof}By the triangle inequality,
 \eqref{eq:convergence_uHMC_bound_W_1}, \Cref{assum:distance}, and \Cref{theo:wasserstein_2_bound},
\begin{eqnarray*}
\wassersteinLigne[p,\dist]{\pi_{T, \gamma} , \pi} &\leq & 
\wassersteinLigne[p,\dist]{\pi_{T,\gamma}\rmK_{T,\gamma} , \pi \rmK_{T,\gamma}} + \wassersteinLigne[p,\dist]{\pi \rmK_{T,\gamma} , \pi \rmK_{T}}\\
&\le &(1-c)\wassersteinLigne[p,\dist]{\pi_{T,\gamma} , \pi}+
C_\dist  {\gamma^2} \, \Ltt^{-1}
\rme^{ \mu T}   M_{\Hamiltonian}^{1/2} .
\end{eqnarray*}
The assertion follows by rearranging.
\end{proof}

{\Cref{theo:wasserstein_2_bound_final_hmc} is a counterpart to \Cref{theo:bias_final_ula_1}.
It would also be possible to state a corresponding counterpart to \Cref{theo:bias_final_ula_2}. We do not consider such a result here
since in contrast to the results for ULA, it does not seem to provide a substantial improvement compared to \Cref{theo:wasserstein_2_bound_final_hmc}.}\smallskip

Similarly as in the results above, we see that
if $\Ltt$ and $c$ are independent of the dimension $d$, then 
the dimension dependence of the upper bounds is determined 
completely by the key quantities $C_\dist$ and $M_\Hamiltonian$.
In Section \ref{example:EulerhighD}, we will show that under appropriate assumptions and
depending on the structure of the model, the resulting dimension dependence 
for upper bounds of the standard $\rml^1$ Wasserstein distance $\mathbf W_1(\pi_\gamma ,\pi )$
is then either of order $O(\gamma^2d^{1/2})$ or of order $O(\gamma^2d)$.

\tcr{
  We conclude this section by noting that, similar to \Cref{theo:bias_final_ula_2} for ULA, \Cref{theo:wasserstein_2_bound_final_hmc} combined with the convergence of uHMC and xHMC obtained in \cite{BoEbZi2020}, gives complexity bounds for uHMC to achieve a precision $\varepsilon$ in Wasserstein distance. 
  }

\begin{remark}[Total variation bias for uHMC] Quantifying the TV bias for uHMC is more involved than for Euler-Maruyama discretizations. Corresponding results are derived in the paper \cite{bourabee2021mixing} that has been prepared in parallel to this work.
\end{remark}



\section{Accuracy in high dimension: Examples and applications}
\label{sec:appl-exampl}

We now analyze the dimension dependence of the bounds we obtain for ULA and uHMC when applied to a target probability measure with density with respect to the Lebesgue measure of the form \eqref{eq:density_pi}.
{At the end of this section we also discuss the relation of our results
for unadjusted MCMC methods to mixing time bounds for the corresponding
methods with Metropolis adjustment.}\smallskip

{Recall the definitions of the constants $C_\dist$ from  \Cref{assum:distance} , $M_{\Langevin}$ from  \Cref{theo:bound_w_2_euler}, $\tM_{\Langevin}$ from  \Cref{theo:bound_w_2_euler_altern}, and $M_{\Hamiltonian}$ from \Cref{theo:wasserstein_2_bound}.}

\subsection{Accuracy of ULA and uHMC}\label{example:EulerhighD}


Suppose that $\tbg = b$ for all $\gamma \in \ocint{0,\bgamma}$, and
assume that  $b$ satisfies \Cref{as:b_lip} and there exist $K,J\in (0,\plusinfty)$ such that for any $x\in \rset^d$, 
\begin{align}
\label{eqLaplace}
\left|\Delta b(x)\right|^2 & = 
\sum_{i,j=1}^d\langle  \partial^2_{ii}b(x),\partial^2_{jj}b(x)\rangle
=\sum_{i,j,k=1}^d  \partial^2_{ii}b_k(x)\, \partial^2_{jj}b_k(x)
\ \le K\eqsp \eqsp,\\
\label{eqFronbenius2D}
\left\| \rmD^2 b(x)\right\|^2_{\mathrm{F}} & 
=\sum_{i,j=1}^d\left| \partial^2_{ij}b(x)\right|^2
 = \sum_{i,j,k=1}^d\left| \partial^2_{ij}b_k(x)\right|^2
\ \le J \eqsp ,
\end{align}
Then, one easily verifies that $M_3=M_5=0$, 
$$
M_1\ \le  2 K\, +\, 4L^2\left(|b(0)|^2+L^2\int |x|^2\, \rmd \pi\right),\quad M_2\ \le dL^2,\quad\text{and}\quad
M_4\ \le J\eqsp .
$$
It is natural to assume that $|b(0)|^2$ and $\int |x|^2\, \rmd \pi$ are of order $O(d)$.
Then, if for a family of models with varying dimension, the Lipschitz constant $L$ is of order $O(1)$ and the constants $K$ and $J$ are of order $O(d)$,
then the constant $ M_{\Langevin}$ in Theorem \ref{theo:bias_final_ula_1} 
and the constant $ M_{\Hamiltonian}$ in \Cref{theo:wasserstein_2_bound_final_hmc} are of order $O(d)$.
If we assume additionally that $A$, $c$, $\psi$ and $C_\dist$
can be chosen independently of the dimension then 
the upper bounds in \eqref{eq:bias_final_ula_1a} and \eqref{eq:bias_final_ula_1b}
are of order $O(d^{1/2}\gamma )$. Similarly, if we assume that $c$ and $C_\dist$
can be chosen independently of the dimension then 
the upper bound in \eqref{eq:bias_final_uHMC_1a} is of order $O(d^{1/2}\gamma^2 )$.  As remarked in the introduction,
these orders are sharp even in the product case with $\dist (x,y)=|x-y|$.\smallskip

Similarly, one verifies that $\tM_5 = 0$,
\begin{align}
\tM_1 &\le  2K\, +\, 4L^2\left(|b(0)|^2+\int |x|^2\, \rmd \pi_\gamma\right) \eqsp,\quad\tM_2\ \le dL^2\eqsp,\\
  \tM_3 &\le  2 K\, +\, 4L^2\left(|b(0)|^2+\sup_{t\ge 0}\int |x|^2\, \rmd(\pig\rm\rmP_t)\right)\eqsp,\quad\tM_4\ \le dL^2\eqsp,\\
  \tM_6 &\le  2L^2\left(|b(0)|^2+\int |x|^2\, \rmd \pi_\gamma\right) \eqsp.
\end{align}
Again, it can be verified under weak assumptions that these constants, and hence $ \tM_{\Langevin}$ in Theorem \ref{theo:bias_final_ula_2} are of order $O(d)$ if $L$ is of order $O(1)$ and $K$ is of order $O(d)$, see Remark \ref{rem:momentbounds} below. \smallskip

The 
main constraint for corresponding bounds with optimal dimension dependence is the
assumption that $K$ is of order $O(d)$. Note that this assumption is trivially satisfied in
the Gaussian case (i.e., $b$ is linear), and also in the product case (i.e., $b(x) = (b_1(x_1),\ldots,b_d(x_d))$), provided a uniform bound on the components holds. More generally, 
it holds in several important classes of models that are frequently studied in applications,
including the following:
\begin{enumerate}
\item[(i)] \emph{Finite-range graphical models \cite{Jordan}.} Suppose that for $d\in\mathbb N$, there exists a graph $(V_d,E_d)$
with vertex set $V_d=\{ 1,2,\ldots ,d\}$ and maximal degree $n\in\mathbb N$ such that
$\partial^2_{ij}b_k=0$ whenever $i\neq k$ and $\{ i,k\}\not\in E_d$. Then $\left\langle \partial^2_{ii}b,\partial^2_{jj}b\right\rangle$ 
and $\partial^2_{ij}b$
vanish if $i$ and $j$ do not have a common neighbour, and thus \eqref{eqLaplace} and \eqref{eqFronbenius2D} are satisfied with
$K=dn^2\sup_{i,x} |\partial^2_{ii}b(x)|^2$ and
$J=dn^2\sup_{i,j,x} |\partial^2_{ij}b(x)|^2$.
\item[(ii)] \emph{Mean-field interactions \cite{BouRabeeSchuh}.} Suppose that there exists a finite constant $C$ such that 
$  |\partial^2_{ii}b_i(x)|\le C$ for all $i$, and
$  |\partial^2_{ij}b_k(x)|\le C/ d$ for all $i,j,k$ such that $k\neq i$. Then \eqref{eqLaplace} and
\eqref{eqFronbenius2D} are satisfied with 
$K=J=2C^2d$.
\item[(iii)] \emph{Perturbations and composition.} If \eqref{eqLaplace} and \eqref{eqFronbenius2D} are satisfied with constants $K_1$, $J_1$, and $K_2$, $J_2$ of order $O(d)$ for
two drift functions $b_1$ and $b_2 $, then a corresponding condition holds for $b=b_1+b_2$.
\end{enumerate}
It is also possible to verify a corresponding dimension dependence of $K$ and $L$ under locally uniform bounds 
on the derivatives of $b$, combined with an appropriate drift condition. However, although bounds with optimal
dimension dependence hold for many important models, in general, assuming that the
first two partial derivatives of $b$ are uniformly bounded, one can only ensure that
$K$, $J$, $M_{\Langevin} $,  $ \tM_{\Langevin}$ and $ M_{\Hamiltonian}$  are of order $O(d^2)$. In this general case,
the order of
the upper bounds in Theorems \ref{theo:bias_final_ula_1},
\ref{theo:bias_final_ula_2} and \ref{theo:wasserstein_2_bound_final_hmc} differs from the one in the product case by a factor $d^{1/2}$.

\begin{remark}\label{rem:momentbounds}
We briefly comment on how to obtain bounds on the constants $\tM_i$, $i=1,2,\ldots ,5$, in \Cref{theo:bound_w_2_euler_altern} and \Cref{theo:bias_final_ula_2} in the case where the derivatives 
of $b$ are not uniformly bounded. In this case, one requires upper bounds on
$\int W\,  \rmd \pi_{\gamma}$ uniformly in $\gamma \in \ocint{0,\bgamma}$ and on
$\int W\, \rmd(\pi_\gamma \rmP_t)=\int \rmP_tW\, \rmd \pi_{\gamma}$ uniformly in $\gamma \in \ocint{0,\bgamma}$ and $t \geq 0$ for $W: \rset^d \to \rset_+$. Such upper bounds can be established if $W$ satisfies Lyapunov conditions for the transition kernels $\rmR_{\gamma}$, $\gamma \in \ocint{0,\bgamma}$ and the generator $\generatorL$. More specifically, assume that there exist  $\mathrm{a} \in \coint{0,1}$ and $\mathrm{b}\geq 0$ such that for any $\gamma \in \ocint{0,\bgamma}$ and $x \in \rset^d$,
\begin{equation}
\rmR_{\gamma} W(x) \leq \mathrm{a}^{\gamma} W(x) + \mathrm{b}\gamma \eqsp .
\end{equation}
Then by \cite[Lemma 1]{durmus:moulines:2016}, we obtain that for any $\gamma \in \ocint{0,\bgamma}$,
\begin{equation}
  \label{eq:12}
  \rmR_{\gamma}^{\ceil{1/\gamma}} W(x) \leq \mathrm{a} W(x) + \mathrm{b}\mathrm{a}^{-\bgamma}/\log(1/\mathrm{a}) \eqsp, 
\end{equation}
and therefore by \cite[Theorem 19.4.1]{douc:moulines:priouret:soulier:2018},
\begin{equation}
  \label{eq:bound_int_pi_gamma_unif}
  \int W\,  \rmd \pi_{\gamma}\leq \ \mathrm{b}\mathrm{a}^{-\bgamma}/\{(1-\mathrm{a})\log(1/\mathrm{a})\} \eqsp.
\end{equation}
Similarly, if $W$ is twice continuously differentiable, and there exist   $\alpha \in (0,\plusinfty )$ and $\beta \in [ 0,\plusinfty )$ such that for any $x \in \rset^d$,
\begin{equation}
  \label{eq:11}
  \generatorL W(x) \leq -\alpha W(x) + \beta \eqsp,
\end{equation}
then by It\=o's formula it holds that for any $t \geq 0$ and $x \in \rset^d$,
\begin{equation}
  \label{eq:bound_int_P_t_unif_t}
  \rm\rmP_t W(x) \leq \rme^{-\alpha t} W(x) + (1-\rme^{-\alpha t}) \beta/\alpha  \eqsp.
\end{equation}
Combining \eqref{eq:bound_int_pi_gamma_unif} and \eqref{eq:bound_int_P_t_unif_t} implies upper bounds 
for $\int \rmP_tW\, \rmd \pi_{\gamma}$.
\end{remark}

\begin{example}[Euler scheme for asymptotically contractive drifts]\label{example:Euler}
Suppose in addition to the assumptions made above that there exists 
$\mathcal K,\mathcal R\in (0,\plusinfty )$ such that for any $x,y\in\mathbb R^d$ with $|x-y|\ge\mathcal R$, 
\begin{equation}
\label{eq:asycontr} \langle b(x)-b(y),x-y\rangle\ \le -\mathcal K|x-y|^2.
\end{equation}
Then as a consequence of \cite[Corollary 2]{eberle2016}, as well as \cite[Theorem 2.12]{EberleMajka} and \Cref{lem:bound_n_Q}, respectively, there exists an explicit distance function 
$\dist $ on $\mathbb R^d$ and explicit constants $c,m,\bgamma\in (0,\plusinfty )$ 
that depend only on $L$, $\mathcal K$ and $\mathcal R$ but not on the dimension $d$
such that for $p=1$, Conditions \eqref{eq:convergence_ula_bound_W_2} and \eqref{eq:convergence_ula_bound_W_1} are satisfied with $A=1$ for all $\gamma\in (0,\bgamma ]$, and for all $x,y\in\mathbb R^d$, 
\begin{equation}
\label{eq:equivalencemetics} m|x-y|\ \le \dist (x,y)\ \le  |x-y|.
\end{equation}
Hence in this case, Theorems  \ref{theo:bias_final_ula_1} and \ref{theo:bias_final_ula_2} show that
\begin{eqnarray*}\mathbf W_1(\pi_\gamma ,\pi ) &\le & m^{-1}\wassersteinLigne[1,\dist]{\pi_{\gamma} , \pi} \leq \      \gamma\,  \rme^{1+\lambda\gamma }\left(\frac{\lambda}{c}+1\right)   m^{-1} M_{\Langevin}^{1/2},\\
\mathbf W_1(\pi_\gamma ,\pi ) &\le & m^{-1}\wassersteinLigne[1,\dist]{\pi_{\gamma} , \pi} \leq \     \gamma\,  \rme^{1+(1+\kappa)\gamma }\left(\frac{1+\kappa}{c}+1\right)   m^{-1} \tM_{\Langevin}^{1/2},
\end{eqnarray*}
respectively.
Since for fixed values of $L$, $K$ and $\mathcal R$, all the other constants are dimension-free, the dimension dependence of these bounds is completely determined by $\tM_{\Langevin}^{1/2}$ and $  M_{\Langevin}^{1/2}$. As pointed out above, the resulting bounds
are of order $O(\gamma d^{1/2})$ for models of type (i), (ii) or (iii), but only
of order $O(\gamma d)$ for general models. Consequently, in order to 
achieve a given bound on the asymptotic bias for general Lipschitz continuous functions, the step size $\gamma$ in the unadjusted Langevin algorithm has to be chosen
of order $O(d^{-1/2})$ for ``nice'' models, and of order $O(d^{-1})$ for general models. Regarding the total variation bounds established in \Cref{theo:bound_tv}, we get bounds
of order $O(\gamma \log(\gamma^{-1}) d^{1/2})$ for models of type (i), (ii) or (iii), but only
of order $O(\gamma \log(\gamma^{-1}) d)$ for general models.
\end{example}

\begin{remark}[$\mathbf W_2$ bounds]
In the globally contractive case where the conditions in Example \ref{example:Euler} are satisfied with $\mathcal R=0$, one also obtains corresponding
bounds for $p=2$ and $\dist (x,y)=|x-y|$ and we get back \cite[Corollary 9]{durmus:moulines:bernoulli}. 
\end{remark}


\begin{example}[Unadjusted HMC for asymptotically contractive drifts]\label{example:uHMCasympcontr}
Suppose again that there exist 
$\mathcal K,\mathcal R\in (0,\plusinfty )$ such that Condition 
\eqref{eq:asycontr} is satisfied for all $x,y\in\mathbb R^d$ with $|x-y|\ge\mathcal R$. Then by \cite[Theorem 2]{BouRabeeSchuh} and \Cref{lem:bound_n_Q}, there exist an explicit distance function 
$\dist $ on $\mathbb R^d$ and explicit constants $c,m,\bgamma\in (0,\plusinfty )$ 
that depend only on $L$, $\mathcal K$ and $\mathcal R$ but not on the dimension $d$
such that 
\eqref{eq:equivalencemetics} holds, and
Condition \eqref{eq:convergence_uHMC_bound_W_1} is satisfied for $p=1$  and all $\gamma\in (0,\bgamma ]$.
Hence by Theorem  \ref{theo:wasserstein_2_bound_final_hmc}, 
$$\mathbf W_1(\pi_\gamma ,\pi ) \ \le \ m^{-1}\wassersteinLigne[1,\dist]{\pi_{\gamma} , \pi} \leq \     
  {\gamma^2}  \,c^{-1} \Ltt^{-1}
\rme^{ \mu T}   m^{-1}M_{\Hamiltonian}^{1/2} .
$$
For fixed values of $L$, $K$ and $\mathcal R$, the dimension dependence of this bound is completely determined by $  M_{\Hamiltonian}^{1/2}$. As shown above, the resulting bound
is of order $O(\gamma^2 d^{1/2})$ for models of type (i), (ii) or (iii), but only
of order $O(\gamma^2 d)$ for general models. Consequently, in order to 
achieve a given bound on the asymptotic bias for general Lipschitz continuous functions, the step size $\gamma$ in unadjusted HMC has to be chosen
of order $O(d^{-1/4})$ for ``nice'' models, and of order $O(d^{-1/2})$ for general models.
\end{example}

\begin{example}[Euler scheme for weakly interacting systems]\label{example:Eulermeanfield}
Another class of models for which dimension-free bounds for convergence to
equilibrium are available are mean-field models and more general interacting
systems with weak interactions, see \cite{eberle2016}. Suppose that $d=nk$
with $n,k\in\mathbb N$, and assume that there exist twice continuously differentiable functions $b_0:\mathbb R^k\to\mathbb R^k$ and $\gamma:\mathbb R^d\to\mathbb R^k$ such that for $x=(x_1,x_2,\ldots ,x_n)\in \mathbb R^{nk}$,
$$b(x)= (b_1(x),b_2(x),\ldots ,b_n(x))\quad\text{ with }\quad 
b_i(x)= b^0 (x_i)+\gamma_i(x).$$
We assume that $b^0 $ satisfies corresponding conditions as $b$ in Example \ref{example:Euler} with constants $\mathcal R$, $L$ and $K$, and we consider the $\ell_1$ metric 
$$\dist (x,y)= \sum_{i=1}^n\dist_0 (x_i,y_i) $$
where the distance function $\dist_0$ on $\mathbb R^k$ is chosen as in Example \ref{example:Euler} (but for $b^0$ instead of $b$). Then by
\cite[Theorem 7]{eberle2016}, there exist  $c,\epsilon >0$ that depend only on  $\mathcal R$, $L$ and $K$ such that Condition \eqref{eq:convergence_ula_bound_W_2} is satisfied with $A=1$ whenever 
\begin{equation}\label{weak_interaction}
\sum_{i=1}^n|\gamma_i(x)-\gamma_i(y)|\ \le \epsilon\, \sum_{i=1}^n|x_i-y_i|\qquad\text{for all }x,y\in\mathbb R^d.
\end{equation}
Since $\dist (x,y)\le \sum_{i=1}^n |x_i-y_i| \le n^{1/2} |x-y|$,
Theorem  \ref{theo:bias_final_ula_2} implies the bound
$$
\mathbf W_{1,\ell^1}(\pi_\gamma ,\pi ) \ \le \ m^{-1}\wassersteinLigne[1,\dist]{\pi_{\gamma} , \pi} \leq \     \gamma\,  \rme^{1+(1+\kappa)\gamma }\left(\frac{1+\kappa}{c}+1\right)   m^{-1} n^{1/2}\tM_{\Langevin}^{1/2},
$$
where $\kappa$ is the constant in \Cref{ass:one_side_lip_b}. If we assume
that $\kappa$ does not depend on the number $n$ of components,
then the upper bound for $\mathbf W_{1,\ell^1}(\pi_\gamma ,\pi )$ depends on $n$ only through $\tM_{\Langevin}^{1/2}$. If we assume additionally that
there exists a finite constant $C$ such that $|\Delta b_i|\le C$ for all
$i\in\{ 1,2,\ldots ,n\}$ then Condition \eqref{eqLaplace} is satisfied with
$K=C^2n$, and hence the upper bound for $\mathbf W_{1,\ell^1}(\pi_\gamma ,\pi )$ is of order $O(\gamma n)$.
This is the optimal order in the product case where $\gamma\equiv 0$. However, the assumptions above are satisfied in more general situations, including for example mean-field models of McKean-Vlasov type with weak interactions where $b_i(x)=-\nabla V(x_i)+\delta n^{-1}\sum_{j\neq i} \nabla W(x_j-x_i)$ for sufficiently 
regular confinement and interaction potentials $V$ and $W$ and a sufficiently 
small coupling parameter $\delta$. On the other hand, for large $\delta$,
these models can exhibit phase transitions. In that case, because of the non-uniqueness of invariant measures for the limiting McKean-Vlasov equation, $\mathbf W_{1,\ell^1}(\pi_\gamma ,\pi )$ will usually degenerate rapidly in high dimensions.
\end{example}

\begin{example}[Unadjusted HMC for mean-field systems]\label{example:uHMCmeanfield}
Corresponding statements as in Example \ref{example:Eulermeanfield} also 
hold for unadjusted Hamiltonian Monte Carlo applied to mean-field models, except 
that the order in $\gamma$ improves from $O(\gamma )$ to $O(\gamma^2)$.
We refer to Bou-Rabee and Schuh \cite{BouRabeeSchuh} for a detailed
analysis of this setup.
\end{example}

\subsection{The Gaussian case}\label{sec:gauss}

For a standard normal target distribution $\pi (\rmd x )=(2\pi)^{-d/2}\exp (-U(x))\, \rmd x$ 
with $U(x)=|x|^2/2$, the asymptotic Wasserstein bias 
of ULA and uHMC can be computed explicitly. The result serves as a benchmark for the general case.

\begin{example}[ULA with standard normal target distribution]\label{example:invuLAGauss}
In this case, for any $x \in \rset^d$, $\tbg(x) = -x$.
It is easy to show that for ${\gamma}\in (0,2)$, \Cref{as:r_gamma} is satisfied and the measure $\pi_{\gamma}$ is the zero-mean Gaussian distribution with covariance matrix $(1-\gamma/2)^{-1}\,\Idd$.
Moreover,
it can be shown
that the synchronous coupling given by $(G,(1-{\gamma}/2)^{-1/2}G)$
where $G$ is a $d$-dimensional zero-mean Gaussian random variable with covariance matrix $\Idd$, is an optimal coupling
of the centered normal distributions $\pi$ and $\pi_{\gamma}$ w.r.t.\ 
$\wassersteinD[p]$ for every $p\in [1,\plusinfty )$. Indeed, by rotational symmetry, this follows from the results in the one-dimensional case
\cite{McCann}, noting
that for any coupling, the average $L^p$ distances are
lower bounded by corresponding Wasserstein distances of the one-dimensional
marginal distributions of the radial parts; see also \cite{GivensShortt} for the case $p=2$. Hence 
\begin{eqnarray*}
\wassersteinD[p] (\pi_{\gamma} ,\pi)& =& \mathbb E\left[ \left|({1-\tfrac{\gamma}2})^{-1/2}G-G\right|^p\right]^{1/p}\\
& =& \left| ({1-\tfrac{\gamma}2})^{-1/2}-1\right|\, \mathbb E\left[ \left| G\right|^p\right]^{1/p}\  \geq  C\gamma d^{1/2} \eqsp,
\end{eqnarray*}
for some constant $C > 0$ independent of $\gamma $ and $d$.
\end{example}

\begin{example}[uHMC with standard normal target distribution]\label{example:invuHMCGauss}
If $ U(q)=|q|^2/2$, then a step of the Verlet integrator is given by 
$\Psiverlet (q,p)= (q',p')$, where 
$$ q'= \left(1-\tfrac{{\gamma}^2}{2}\right)q+{\gamma}p,\quad p'= \left(1-\tfrac{{\gamma}^2}{2}\right)p-{\gamma}\left( 1-\tfrac{{\gamma}^2}{4}\right)q.$$
It can be easily verified that for $\gamma\in (0,2)$, this map preserves
the modified Hamiltonian
\begin{equation}
    \label{HamiltonianVerlet}
    H_{\gamma}(q,p)= \tfrac 12\left( 1-\tfrac{{\gamma}^2}{4}\right)|q|^2\, +\, \tfrac 12 |p|^2 \eqsp,
\end{equation}
i.e., $H_{\gamma}\circ\Psiverlet_{\gamma}=H_{\gamma}$. Since $\Psiverlet_{\gamma}$ also 
preserves the Lebesgue measure on $\mathbb R^{2d}$, we see that the probability measure with density proportional to $\exp (-H_{\gamma}(q,p))$ is preserved under $\Psiverlet_{\gamma}$, and also under momentum 
randomizations. The unique invariant probability measure $\pi_{T,\gamma}$
of uHMC in position space is the first marginal of this measure, i.e., for every $T>0$, $\pi_{T,\gamma}$ is the $d$-dimensional zero-mean Gaussian measure with covariance matrix $( {1-\gamma^2/4})^{-1}\,\Idd$.
Therefore, similarly as in  \Cref{example:invuLAGauss}, we obtain
\begin{eqnarray*}
\wassersteinD[p] (\pi_{T,\gamma} ,\pi)
& =& \left| (1-\gamma^2/4)^{-1/2}-1\right|\, \mathbb E\left[ \left\| G\right\|^p\right]^{1/p}\ \geq C\gamma^2 d^{1/q}  \eqsp. 
\end{eqnarray*}
for some constant $C \geq 0$ independent of $\gamma $ and $d$.
%
\end{example}

{
\subsection{Comparison of unadjusted and Metropolis-adjusted MCMC methods}\label{adjustornotadjust}

As an alternative to applying unadjusted MCMC methods, it is very common
to use Metropolis-Hastings (MH) methods where the transition steps of unadjusted MCMC methods can be used  as proposals, \cite{RobertCasella,roberts:tweedie:1996,neal2011mcmc}.

An obvious advantage of the Metropolis-adjustment is that one obtains a 
Markov chain that  exactly preserves  the target distribution, i.e., the
asymptotic bias vanishes. Consequently, one can at least in
principle approximate
the target distribution with arbitrary precision by running
the MH Markov chain for a sufficiently long time. Moreover, the number
of steps required to achieve a given accuracy  $\epsilon >0$ is of
order $\log \left(\epsilon^{-1} \right)$, while for  inexact schemes,
the step size $\gamma$ has to be adjusted to the desired accuracy, resulting
in a complexity of order $\epsilon^{-\alpha}$ where $\alpha =1$ for ULA
and $\alpha =1/2$ for uHMC.
On the other hand, a disadvantage of MH adjustment is that a high
rejection rate can lead to  slow mixing of the Metropolis-adjusted Markov 
chain, while the mixing properties of the unadjusted Markov chain usually remain stable even for larger step sizes (at the cost of introducing an
asymptotic bias in the estimates). Moreover, the non-smooth dependence
of the trajectories of Metropolis-adjusted chains on parameters or initial
data can cause problems  for both the theoretical analysis and 
practical applications such as the estimation of sensitivities.

To compare Metropolis-adjusted and unadjusted MCMC methods it is useful to distinguish two regimes: 
\begin{itemize}
\item[(i)] If the \emph{acceptance rate} of the Metropolis-adjusted scheme is \emph{``sufficiently high''} then one might expect that the
adjusted chain has as  good mixing properties as the unadjusted chain.
However, no proof of such a general fact is known, and moreover,
the acceptance rate may vary considerably in different regions of the state
space. So far, mixing properties for MALA with a step size $\gamma$ of
order $O(d^{-1/2})$ have been proven \emph{only for strongly log-concave distributions} and \emph{for a warm start}, i.e.,
when the initial distribution already has a relative density w.r.t.\ the target distribution
that is bounded by a fixed constant \cite{chewi2021optimal,wu2021minimax}. It is not known
how such a warm start can be generated in practice, and the best available bounds
for a cold or feasible start require a step size $\gamma $ of order $O(d^{-1})$ \cite{lee2020logsmooth,chen2020fast}. Indeed, it can be shown that this order cannot be improved in general \cite{lee2021lower}, although a better dimension dependence may hold for
subclasses of nice models. The existing rigorous upper bounds for HMC with Metropolis-adjustment are even less satisfactory \cite{BoEbZi2020,chen2020fast}.
\item[(ii)] If, on the other hand, the \emph{acceptance rate} of the Metropolis-adjusted
scheme \emph{degenerates} then it can be  easily shown by a conductance argument that the mixing properties and even the relaxation time also degenerate \cite{eberle:lecture:notes:markov}.
Nevertheless, the unadjusted chain will often have good mixing properties
even for large step sizes. Our results show that in this case, approximate
samples produced by the unadjusted chain can sometimes still provide useful information. In particular, for nice models, the asymptotic Wasserstein bias of unadjusted
HMC is well behaved for step sizes $\gamma$ of order $O(d^{-1/4})$,
but it is known that for adjusted HMC, the acceptance probability can degenerate in this case, unless a warm start condition is assumed \cite{eberle:lecture:notes:markov}.
Moreover, Wasserstein distances are not scale invariant, and the bias in estimating integrals $\int_{\rset^d} f\, \rmd\pi$
by inexact MCMC methods depends on the regularity of the function $f$.
Therefore, even if the Wasserstein bias in a certain metric grows with
the dimension, it may be possible to obtain good approximations 
for integrals of well-behaved observables.
Indeed, we have already seen in the introduction that for example for intensive quantities in molecular dynamics simulations, approximations
are sometimes possible even for step sizes that do not depend on the dimension at all.
This shows another important difference between unadjusted and
Metropolis-adjusted schemes: whereas the latter seem to either degenerate
or work well, unadjusted schemes with large step sizes can still produce
a good approximation for nice observables.
\end{itemize}
In practice, it is usually not known how to adjust the step size to obtain reliable estimates. One possibility, arising from the above discussion, might be to run an unadjusted chain with a large step size at the beginning of the simulation, and then reduce the step size until sufficiently high acceptance probabilities for the Metropolis-adjustment are 
achieved (this could be tested empirically), so that the chain can be run
with Metropolis-adjustment from now on  to fine tune the estimates. 
An important question for future research might be to clarify more precisely
what ``sufficiently high'' means and to  rigorously analyze if unadjusted
schemes are indeed able to generate good initial distributions for Metropolis-adjusted methods.}

\section{Proofs} \label{sec:proofs}

{ This section contains the proofs of the main results. In all cases, the main idea is to apply the triangle inequality trick from Lemma \ref{theo:asymptotic_bias_altern}. Then, assuming convergence bounds in Wasserstein distance for
the exact (respectively approximate) dynamics, the asymptotic bias can be quantified
if we can control the Wasserstein distance between the dynamics and the stationary distribution of the scheme (respectively the iterates of the discretization scheme and the target). To this end, we compare the $L^2$ distance between the
exact and approximate dynamics driven by the same noise. While this
is standard in the analysis of numerical schemes for SDE \cite{kloeden2011numerical}, we carefully analyze how discretization errors propagate along the iterations of the scheme to obtain precise bounds with the correct dependence on dimension. 
\smallskip

After briefly reviewing basic facts on Wasserstein bounds, we first prove
the main results for ULA  (Theorems \ref{theo:bias_final_ula_1} and  \ref{theo:bias_final_ula_1}) and then the main result for uHMC  (Theorem
 \ref{theo:wasserstein_2_bound_final_hmc}).}

\subsection{Wasserstein bounds for transition kernels}
\label{sec:proof-crefl}

For the reader's convenience, we recall the proof of the following well-known result.

\begin{lemma}
  \label{lem:bound_n_Q}
Suppose that $\dist : \mathbb R^d\times\mathbb R^d \to \rset_+$ is a lower semicontinuous distance function, and 
$\rmP$ is a Markov transition kernel
on  $(\mathbb R^d,\mathcal B(\mathbb R^d))$. Let $p\in [1,\plusinfty )$.
If there exists  
$\alpha \geq 0$ such that for all $x,y \in \mathbb R^d$,
\begin{equation}
  \label{eq:varphi_contraction}
  \wassersteinLigne[p,\dist]{\updelta_x \rmP , \updelta_y \rmP}\leq
\alpha \, \dist(x,y) \eqsp ,
\end{equation}
then for all $n\in\mathbb N$ and all probability measures $\mu ,\nu\in\mathcal P(\mathbb R^d)$,
\begin{equation}
  \label{eq:varphi_contraction}
  \wassersteinLigne[p,\dist]{\mu \rmP^n , \nu \rmP^n}\leq
\alpha^n \, \wassersteinLigne[p,\dist]{\mu  , \nu } \eqsp .
\end{equation}
\end{lemma}

\Cref{lem:bound_n_Q} is an immediate consequence of the following lemma.

\begin{lemma}
  \label{lem:marginalize_asympt_bias_altern}
Let $\msx$  be a Polish space with Borel $\sigma$-field $\mcx$, 
$p\in [1,\plusinfty )$, and suppose that $\dist : \msx^2 \to \rset_+$ is a lower semicontinuous distance function on $\msx$.  Consider two Markov kernels $\rmQ_1$ and $\rmQ_2$ on $\msx\times \mcx$, and suppose that there exists a measurable function  $\Psi : \msx \times \msx \to \rset_+$ such that for any $x,y \in \msx$,
  \begin{equation}
    \label{eq:lem:marginalize_asympt_bias_altern}
    \wassersteinLigne[p,\dist]{\updelta_x \rmQ_1, \updelta_y \rmQ_2}\leq \Psi(x,y) \eqsp.
  \end{equation}
Then for any $\mu_1,\mu_2 \in\Pens_{p,\dist}(\msx) $,  and for any coupling  $\zeta \in \Gamma(\mu_1,\mu_2)$,
  \begin{equation}\label{conclusion20}
          \wasserstein[p,\dist]{\mu_1 \rmQ_1,  \mu_2 \rmQ_2}\leq \left\{ \int_{\msx\times \msx} \Psi(x,y)^p \,  \zeta (\rmd (x,y))\right\}^{1/p}
          \eqsp. 
  \end{equation}
In particular if $\Psi = \alpha\, \dist$ for $ \alpha \geq 0$, then $          \wasserstein[p,\dist]{\mu_1 \rmQ_1,  \mu_2 \rmQ_2} \leq \alpha \wasserstein[p,\dist]{\mu_1,  \mu_2}$.
\end{lemma}

\begin{proof}
Let $\mu_1,\mu_2 \in \Pens_{p,\dist}(\msx)$ and $\zeta \in \Gamma(\mu_1,\mu_2)$.
  By \cite[Corollary 5.22]{VillaniTransport}, there exists a Markov kernel $K$ on $(\msx \times \msx) \times (\mcx\otimes \mcx)$ such that for any $x,y \in\msx$, the probability measure
  $K((x,y),d(w,z))$ is an optimal coupling of $ \rmQ_1(x,dw)$ and $\rmQ_2(x,dz)$, i.e., 
  $      \wassersteinLigne[p,\dist]{\updelta_x \rmQ_1, \updelta_y \rmQ_2} = \left\{\int_{\msx\times \msx} \dist(w,z)^p K((x,y),\rmd(w,z))\right\}^{1/p}$.
By Fubini's theorem, the probability measure  $\zeta K$ is a coupling of $\mu_1\rmQ_1$ and $ \mu_2 \rmQ_2$. Therefore, by definition of $\wassersteinD[p,\dist]$, Fubini's theorem, and \eqref{eq:lem:marginalize_asympt_bias_altern},
  \begin{eqnarray*}
              \wasserstein[p,\dist]{\mu_1\rmQ_1,\mu_2 \rmQ_2}&\leq &  \left\{\int_{\msx \times \msx} \int_{\msx\times \msx} \dist(w,z)^p\,    K((x,y),\rmd(w,z)) \,\zeta (\rmd(x,y))\right\}^{1/p}\\
              &  \leq &  \left\{\int_{\msx \times \msx} \Psi(x,y)^p \,  \zeta (\rmd(x,y))\right\}^{1/p} \eqsp,
            \end{eqnarray*}
The last statement follows by taking the infimum over $\zeta \in \Gamma(\mu_1,\mu_2)$.
\end{proof}

\begin{proof}[Proof of \Cref{lem:bound_n_Q}]
Applying \Cref{lem:marginalize_asympt_bias_altern} with $\rmQ_1 =\rmQ_2
=\rmP$, $\mu_1 =\mu$ and $\mu_2 = \nu$ yields $   \wassersteinLigne[p,\dist]{\mu{\rmP} , \nu \rmP} \leq
\alpha\wassersteinLigne[p,\dist]{ \mu ,\nu}$. The claim then follows by induction. 
\end{proof}

\subsection{Proofs of \Cref{theo:bound_w_2_euler} and \Cref{theo:bias_final_ula_1}}
\label{sec:proof-ula_One}
We consider a synchronous coupling between the diffusion process
\eqref{eq:sde} and its discretization \eqref{eq:def_euler_maru}.
Let $W_0$ be an $\mathbb R^d$-valued random variable with $\expeLigne{\abs[2]{W_0}} < \plusinfty$ that is independent of the $d$-dimensional Brownian motion $(B_t)_{t \geq 0}$. We define processes $(Y_t)_{t \geq 0}$ and $(\bY_t)_{t \geq 0}$ by $Y_0 = \bY_0 = W_0$,
\begin{equation}
  \label{eq:def_synchronuous}
  \begin{aligned}
     Y_t &= Y_0 + \int_{0}^t b(Y_s) \rmd s  + \sqrt{2} B_t \eqsp,\\
    \bY_{t}& = \bY_0 + \int_{0}^t \tbg(\bY_{\floor{s/\gamma} \gamma}) \rmd s  + \sqrt{2} B_t \eqsp.
  \end{aligned}
\end{equation}
Then $(Y_t)_{t\ge 0}$ is the unique strong solution of the SDE \eqref{eq:sde}
with initial condition $W_0$, and $(\bY_t)_{t \geq 0}$ is the linear interpolation
of the Euler-Maruyama type discretization in the sense that for every $n\in\mathbb N$, $X_n=bY_{n \gamma} $ satisfies the recursion \eqref{eq:def_euler_maru} with independent standard normal random 
variables $G_k$ given by $ G_{k+1} =( B_{(k+1)\gamma} - B_{k\gamma})/\sqrt\gamma$. In particular, for any $n\in\nset$ and $\gamma \in \ocint{0,\bgamma}$, $(Y_{n\gamma},
\bY_{n \gamma})$ is a coupling of the probability measures 
$\nu \rmP_{n \gamma}$ and $\nu \rmR_{\gamma}^n$, where $\nu$ is the law
of the initial value $W_0$ and therefore
\begin{equation}
  \label{eq:bound_diff_wasser_coupling}
  \wasserstein[2]{\nu \rmP_{n \gamma} , \nu \rmR_{\gamma}^n } \
  \leq  \PE^{1/2}\parentheseDeux{\abs[2]{Y_{n\gamma} - \bY_{n\gamma}}}\eqsp.
\end{equation}
Finally, note that if $\nu = \pi$, then by \Cref{assum:diffusion_semigroup_a}, $(Y_t)_{t\ge 0}$
is a stationary process and for any $t \geq 0$, $Y_t$ has distribution $\pi$.



\begin{proof}[Proof of  \Cref{theo:bound_w_2_euler}]
We apply \eqref{eq:bound_diff_wasser_coupling} with $\nu =\pi$.
By \Cref{lem:bound_w_2_euler} below, and a straightforward induction, we obtain that for any $\gamma \in \ocint{0,\bgamma}$ and $n \in \nset$,
  \begin{align}
&    \wasserstein[2]{\pi \rmP_{n\gamma}, \pi\rmR_{\gamma}^n}[2] \leq \expeE{\abs[2]{ Y_{n \gamma}- \overline{Y}_{n \gamma}}} \\
    &  \qquad  \leq \gamma^3  \parentheseDeux{    (2+9\gamma )M_1/6 + 3 M_2 + (2+3\gamma) M_3} \sum_{k=1}^n (1+ 2\lambda_{\Langevin}\gamma )^{n-k}  \\
    & \qquad  \leq  \frac{\gamma^2}{2\lambda_{\Langevin}}  \parentheseDeux{    (2+9\gamma )M_1/6 + 3 M_2 + (2+3\gamma) M_3}     (1+ 2\lambda_{\Langevin}\gamma )^{n} \eqsp. 
  \end{align}
  The proof is concluded using that for any $t \geq 0$, $1+t \leq \rme^{t}$ and $\lambda_{\Langevin} \geq 1$. 
\end{proof}

\begin{lemma}
  \label{lem:bound_w_2_euler}
  Assume \Cref{assum:diffusion_semigroup_a}, \Cref{as:tilde_b_b}, \Cref{as:r_gamma} and \Cref{as:b_lip}.
  Then for any $\gamma \in \ocint{0,\bgamma}$ and  $n \in \nset$,
\begin{multline}
\label{eq:lem:bound_w_2_euler_res}
  \expeE{\abs[2]{ Y_{(n+1) \gamma}- \overline{Y}_{(n+1) \gamma}}}  \leq  (1+2 \lambda_{\gamma}\gamma) \expeE{\abs[2]{ Y_{n \gamma}- \overline{Y}_{n \gamma}}}\\ + \gamma^3 \parentheseDeux{   (2+9\gamma )M_1/6 + 3 M_2 + (2+3\gamma) M_3} \eqsp,
\end{multline}
where
\begin{equation}
  \label{eq:5}
\lambda_{\gamma} =   1+L^2+3 \gamma L^2/2 \eqsp,
\end{equation}
 $(Y_t,\bY_t)_{t \geq 0}$ is defined by \eqref{eq:def_synchronuous} with $W_0$ distributed according to $\pi$, and $M_1,M_2,M_3$ are given by \eqref{eq:def_m_1_m_2_m_3}. 
\end{lemma}

\begin{proof}
For any $k \in \nset$, define $Z_k\eqdef Y_{k \gamma}- \overline{Y}_{k \gamma}$ and let $n\in\mathbb N$, $\gamma \in\ocint{0,\bgamma}$. Then by \eqref{eq:def_synchronuous} and using the decomposition $b(Y_s) -\tbg(\bY_{n\gamma})  = b(Y_s) - b(Y_{n \gamma} +b(Y_{n \gamma} )-\tbg(\bY_{n\gamma})  $, we get 
\begin{eqnarray}
\label{eq:lem:bound_w_2_euler_a_0}
 \lefteqn{ \expeE{\abs[2]{Z_{n+1}}}\  =    \expeE{\abs[2]{Z_n}} +
  \int_{n \gamma}^{(n+1) \gamma} 2 \expeE{\ps{Z_n}{ {b(Y_s) - b(Y_{n \gamma} ) } }}\; \rmd s } \\
  \nonumber
  & &+2\gamma\,  \expeE{\ps{Z_n}{ { b(Y_{n \gamma} )-\tbg(\bY_{n\gamma}) } }} +\expeE{\abs[2]{\int_{n \gamma}^{(n+1) \gamma} \defEns{b(Y_s) - \tbg(\bY_{n\gamma})  } \rmd s}}.
\end{eqnarray}
We now bound the terms on the right hand side.
First, by It\=o's formula, for any $s \ge n \gamma$,
\begin{equation}
  \label{eq:lem:bound_w_2_euler_a_0_1}
  b(Y_s) - b(Y_{n\gamma}) = \int_{n\gamma}^s \generatorL b(Y_u) \rmd u +  \int_{n\gamma}^{s} \langle\nabla b(Y_u) , \rmd B_u\rangle  \eqsp.
\end{equation}
Denote by $(\mcf_t^{B})_{t \geq 0}$ the filtration associated with $(B_t)_{t \geq 0}$. 
Since $M_2 < \plusinfty$, the process $(\int_{0}^{s}\langle \nabla b(Y_u), \rmd B_u\rangle)_{s \ge 0}$ is a $(\mathcal{F}_t^B)_{t \geq 0}$-martingale. Using that $(Z_t)_{t \geq 0}$ is $(\mcf_{t }^B)_{t \geq 0}$-adapted and for any $t \geq 0$, $Y_t$ has distribution $\pi$, and we get by the Cauchy-Schwarz inequality 
\begin{align}
\nonumber  
 2 {  \expeE{\ps{Z_n}{{b(Y_s) - b(Y_{n\gamma}) }  }}} &=
2 {  \expeE{\ps{Z_n}{  \int_{n\gamma}^s \generatorL b(Y_u) \rmd u}}} \\
  \qquad \qquad & \leq    \expeE{\abs[2]{Z_n}} +  \expeE{\abs[2]{\int_{n\gamma}^s \generatorL b(Y_u) \rmd u}}\\
  \qquad \qquad  &  \le   \expeE{\abs[2]{Z_n}} +  (s-n\gamma )^2M_1 \eqsp. \nonumber
\end{align}
Therefore, we get 
\begin{equation} \label{boundterm1} 
 \int_{n \gamma}^{(n+1) \gamma} 2 \expeE{\ps{Z_n}{ {b(Y_s) - b(Y_{n \gamma} ) } }}\; \rmd s \  \le {\gamma}\, \expeE{\abs[2]{Z_n}} + \gamma ^3M_1/3\, .
\end{equation}
Furthermore, using the decomposition $b(Y_{n \gamma} )-\tbg(\bY_{n\gamma})  =  b(Y_{n \gamma}) - \tbg(Y_{n \gamma}) +\tbg(Y_{n\gamma}) - \tbg(\overline{Y}_{n \gamma})$, as well as \Cref{as:tilde_b_b} and \Cref{as:b_lip}, we have
\begin{align}\nonumber
2 \expeE{\ps{Z_n}{ { b(Y_{n \gamma} )-\tbg(\bY_{n\gamma}) } }}  &\le 
  \expeE{\abs[2]{Z_n}} +  \expeE{\abs[2]{b(Y_{n \gamma} )-\tbg(\bY_{n\gamma})}}
  \\ 
&\qquad  \le    \expeE{\abs[2]{Z_n}} + 2\gamma^2M_3+2L^2    \expeE{ \abs[2]{Z_n }}  \eqsp, \label{boundterm2}
\end{align}
where $M_3$ is defined in \eqref{eq:def_m_1_m_2_m_3}.
Using $b(Y_s) - \tbg(\bY_{n\gamma}) = b(Y_s) - b(Y_{n\gamma}) + b(Y_{n \gamma}) - \tbg(Y_{n \gamma}) +\tbg(Y_{n\gamma}) - \tbg(\overline{Y}_{n \gamma})$ and \eqref{eq:lem:bound_w_2_euler_a_0_1}, we get that 
\begin{eqnarray}
\nonumber
\lefteqn{ \expeE{\abs[2]{\int_{n \gamma}^{(n+1) \gamma} \defEns{b(Y_s) - \tbg(\bY_{n\gamma}) } \rmd s}}  \leq 
3   \expeE{\abs[2]{\int_{n \gamma}^{(n+1) \gamma} \defEns{b(Y_s) - b(Y_{n\gamma}) } \rmd s}}}
  \\
 \nonumber  
  & &\qquad \quad  +   3  \gamma^2 \expeE{\abs[2]{b(Y_{n\gamma})- \tbg(Y_{n\gamma}) }}
    + 3  \gamma^2 \expeE{ \abs[2]{\tbg(Y_{n\gamma}) - \tbg(\overline{Y}_{n\gamma}) }}
  \\
  \nonumber
  &  \leq &
6   \expeE{\abs[2]{\int_{n \gamma}^{(n+1) \gamma} \int_{n\gamma}^s \generatorL b(Y_u) \rmd u \rmd s}+\abs[2]{\int_{n \gamma}^{(n+1) \gamma}  \int_{n\gamma}^{s} \langle\nabla b(Y_u), \rmd B_u\rangle \rmd s}} 
  \\   
  & &\qquad \quad  +   3  \gamma^4M_3 
+ 3  \gamma^2L^2    \expeE{ \abs[2]{Z_n }}\eqsp .  \label{eq:lem:bound_w_2_euler_a}
\end{eqnarray}
Using the Cauchy-Schwarz inequality and that for any $t \geq 0$, $Y_t$ has distribution $\pi$, we have
\begin{equation}\label{eq:lem:bound_w_2_euler_b}
 \expeE{\abs[2]{\int_{n \gamma}^{(n+1) \gamma} \int_{n\gamma}^s
  \generatorL b(Y_u) \rmd u \rmd s}}  \le \defEns{\int_{n \gamma}^{(n+1) \gamma} \int_{n\gamma}^s  \rmd u \rmd s}^2M_1 = \frac{\gamma^4}{4}M_1.
\end{equation}
%
Similarly, using  the Cauchy-Schwarz inequality and It\=o's isometry, we obtain
\begin{eqnarray}\nonumber
\lefteqn{\expeE{\abs[2]{\int_{n \gamma}^{(n+1) \gamma}  \int_{n\gamma}^{s} \langle\nabla b(Y_u), \rmd B_u\rangle \rmd s}}}\\ & \leq &  \gamma     {\int_{n \gamma}^{(n+1) \gamma} \int_{n\gamma}^{s} \expe{ \Tr\parenthese{ \nabla b(Y_u) \nabla b(Y_u)^{\transpose}}} \rmd u \rmd s} 
  \label{eq:lem:bound_w_2_euler_b_a}
=     \frac{\gamma^3}2 M_2 \eqsp .
\end{eqnarray}
%
The proof then follows from 
combining  
\eqref{boundterm2}, \eqref{eq:lem:bound_w_2_euler_a}, \eqref{eq:lem:bound_w_2_euler_b} and \eqref{eq:lem:bound_w_2_euler_b_a} in \eqref{eq:lem:bound_w_2_euler_a_0}.
\end{proof}

{\begin{remark}
An alternative way to arrive at bounds as in \Cref{lem:bound_w_2_euler}
is through stochastic interpolation formulae \cite{DelMoralSingh20,DelMoralSingh22}. These provide exact expressions
for the difference of two stochastic flows. In the simple scenario considered here, they seem to lead to similar bounds as above. However, the interpolation approach 
might be helpful in analyzing discretizations of stochastic differential equations with non-constant diffusion coefficients.
\end{remark}}

\begin{proof}[Proof of \Cref{theo:bias_final_ula_1}]
The result is a direct consequence of \Cref{theo:bound_w_2_euler}
and the inequalities in \eqref{eq:accuracyinvariantmeasure2} and \eqref{eq:accuracyinvariantmeasure3}.
\end{proof}



\subsection{Proofs of  \Cref{theo:bound_w_2_euler_altern} and \Cref{theo:bias_final_ula_2}}
\label{sec:proofs-crefth}  
Similarly as above, we consider  $(Y_t,\bY_t)_{t \geq 0}$  defined by  \eqref{eq:def_synchronuous}, but now with $W_0$ distributed according to $\pig$. Then since $\pig$ is invariant for $\rmR_{\gamma}$ by \Cref{as:r_gamma}, the process $(X_n=bY_{n\gamma})_{n \in\nset}$ is stationary  and for any $n \in\nset$ and $\gamma \in\ocint{0,\bgamma}$, $\bY_{n\gamma}$ has distribution  $\pig$.

\begin{proof}[Proof of \Cref{theo:bound_w_2_euler_altern}]
  By \Cref{lem:bound_w_2_euler_altern} below and since $Y_0 = \bY_0$, we have by a straightforward induction that for any $\gamma \in \ocint{0,\bgamma}$
and $n\in\mathbb N$,
$$
\expeE{\abs[2]{ Y_{n\gamma}- \overline{Y}_{n\gamma}}}\leq \  (1+2\kappa) \int_{0} ^{n\gamma}   \expeE{\abs[2]{ Y_{s}- \overline{Y}_{s}}} \rmd s+n\gamma^3 \tM_{\Langevin} \eqsp.$$
Therefore, we get for any $t \geq 0$, using \Cref{lem:bound_w_2_euler_altern} again,
\begin{equation}
  \label{eq:8}
  \expeE{\abs[2]{ Y_{t}- \overline{Y}_{t}}}\leq \  (1+2\kappa) \int_{0} ^{t}   \expeE{\abs[2]{ Y_{s}- \overline{Y}_{s}}} \rmd s+(t\gamma^2+\gamma^3)\ \tM_{\Langevin} \eqsp.
\end{equation}
  By Grönwall's inequality, and since $s \leq \rme^s$, we obtain that for any $t \geq 0$,
  \begin{equation}
    \expeE{\abs[2]{ Y_{t}- \overline{Y}_{t}}}\leq  \rme^{1+2\kappa t} \gamma^2(t +\gamma)\tM_{\Langevin} \leq \rme^{2(1+\kappa)t} \gamma^2(1+\gamma)\tM_{\Langevin}  \eqsp.
  \end{equation}
The proof is then completed using \eqref{eq:bound_diff_wasser_coupling}.  
\end{proof}

We preface the proof of \Cref{lem:bound_w_2_euler_altern} by a technical result.
\begin{lemma}
  \label{bound_moment_bY}
  Assume \Cref{as:r_gamma} and let $f :\rset^d\times\rset^d \to \rset_+$ be a measurable function. Let  $(\bY_t)_{t \geq 0}$ be defined by \eqref{eq:def_synchronuous} with $W_0$ distributed according to $\pig$. Then for any $n \in \nset$, $\gamma \in \ocint{0,\bgamma}$ and  $u \in \ccint{0,\gamma}$,
  \begin{equation}
    \expe{f(\bY_{n\gamma},\bY_{n\gamma +u})} = \int_{\rset^d}\int_{ \rset^d}  {  f(x,x+u \tbg(x) + (2u)^{1/2}z) }\;  \varphibf_d(z)\, \rmd z\; \pig  (\rmd x)  \eqsp .
  \end{equation}
\end{lemma}
\begin{proof}
  Let $n \in \nset$, $\gamma \in \ocint{0,\bgamma}$ and  $u \in \ccint{0,\gamma}$.
By definition,
  $\bY_{n\gamma +u} = \bY_{n\gamma} + u \tbg(\bY_{n\gamma}) + \sqrt 2 (B_{n\gamma +u}- B_{n
    \gamma})$. Then, since $\bY_{n\gamma}$ is $(\mathcal{F}^B_{t})_{t \leq n\gamma}$-measurable, where $(\mathcal{F}^B_{t})_{t \leq n\gamma}$ is the filtration generated by $(B_t)_{t \geq 0}$. By the Markov property of the Brownian
  motion, the increment  $B_{n\gamma +u}- B_{n
    \gamma}$ is independent of $\bY_{n\gamma}$ and therefore, we get
  \begin{equation}
    \expe{f(\bY_{n\gamma},\bY_{n\gamma +u}) | \bY_{n\gamma}} = \int_{ \rset^d}  {  f(\bY_{n\gamma},\bY_{n\gamma}+u \tbg(\bY_{n\gamma}) + (2u)^{1/2}z) }\;  \varphibf_d(z)\, \rmd z  \eqsp .
  \end{equation}  
  The proof is then completed using that  $ \bY_{n\gamma}$ has distribution $\pi_\gamma$.
\end{proof}

\begin{lemma}
  \label{lem:bound_w_2_euler_altern}
   Assume \Cref{assum:diffusion_semigroup_a}, \Cref{as:tilde_b_b}, \Cref{as:r_gamma} and \Cref{ass:one_side_lip_b}.  Let $(Y_t,\bY_t)_{t \geq 0}$ be defined by \eqref{eq:def_synchronuous} with $W_0$ distributed according to $\pig$.
  Then for any $\gamma \in \ocint{0,\bgamma}$, $n \in \nset$, $t \in \ccint{n \gamma, (n+1)\gamma}$,
\begin{eqnarray*}
\label{eq:lem:bound_w_2_euler_res_altern}
\expeE{\abs[2]{ Y_{t}- \overline{Y}_{t}}}& \leq &   \expeE{\abs[2]{ Y_{n \gamma}- \overline{Y}_{n \gamma}}} + (1+2\kappa) \int_{n \gamma} ^{t}   \expeE{\abs[2]{ Y_{s}- \overline{Y}_{s}}} \rmd s \\
&&   +\gamma^3\left( \frac 16  \tM_1    +\frac 12\gamma^{1/2}\tM_2^{1/2}\tM_3^{1/2} + \frac{1}{\sqrt 2}\tM_2^{1/2}\tM_4^{1/2}+ \frac 12\tM_5\right)\eqsp ,
\end{eqnarray*}
where  $\tM_i$, $i\in\{1,\ldots,5\}$, are defined in  \Cref{theo:bound_w_2_euler_altern}.
\end{lemma}

\begin{proof}
  Let $\gamma \in \ocint{0,\bgamma}$ and $n \in \nset$ and for any $t \geq 0$, $Z_t = Y_t - \bY_t$. By \eqref{eq:def_synchronuous}, almost surely it holds 
  \begin{equation}
    \label{eq:expre_theta_langevin}
\rmd  Z_t / \rmd t  = \ b(Y_t) - \tbg(\bY_{n\gamma})\qquad\text{for } t \in [n \gamma,(n+1)\gamma) \eqsp. 
  \end{equation}
  Therefore and by \Cref{ass:one_side_lip_b}, we have for any $t \in \ccint{n \gamma,(n+1)\gamma}$,
  \begin{align}
    \nonumber
    \abs[2]{Z_{t}} &= \abs[2]{Z_{n\gamma}} + 2\int_{n\gamma}^t \ps{Z_s}{b(Y_s)-\tbg(\bY_{n \gamma}) } \rmd s\\
        \nonumber
                         & = \abs[2]{Z_{n\gamma}} + 2\int_{n\gamma}^t \ps{Z_s}{b(Y_s)-b(\bY_{s}) + b(\bY_{s})-\tbg(\bY_{n \gamma}) } \rmd s \\
    \label{eq:0:lem:bound_w_2_euler_altern}
    & \leq \abs[2]{Z_{n\gamma}} + 2 \kappa\int_{n\gamma}^t \abs[2]{Z_s} \rmd s  + \mathrm{B}_1+ \mathrm{B}_2 \eqsp,\qquad\text{where}
  \end{align}
  \begin{equation}
    \label{eq:def_I_2}
\mathrm{B}_1   =  \int_{n\gamma}^t \ps{Z_s}{b(\bY_{s}) -b(\bY_{n\gamma})  } \rmd s, \  \mathrm{B}_2= \int_{n\gamma}^t \ps{Z_s}{ b(\bY_{n\gamma}) -\tbg(\bY_{n \gamma}) } \rmd s . 
\end{equation}
We now bound  $\expe{\mathrm{B}_1} $ and $\expe{\mathrm{B}_2}$. Let $s \in \ccint{n \gamma,(n+1)\gamma}$. We
first give a bound on $\expeLigne{\absLigne[2]{ b(\bY_{s}) -b(\bY_{n\gamma}) }}$. By \eqref{eq:def_synchronuous} and It\=o's formula, 
\begin{equation}
  \label{eq:ito_b_bY}
 b(\bY_{s}) -b(\bY_{n\gamma})  = \int_{n\gamma}^s \generatorD b(\bY_{n \gamma}, \bY_u)  \rmd u + \int_{n\gamma}^s \langle\nabla b(\bY_u) ,\rmd B_u\rangle \eqsp.
\end{equation}
Therefore, we obtain using the Cauchy-Schwarz inequality and  It\=o's isometry,
\begin{eqnarray*}
 \lefteqn{\expe{ \abs[2]{b(\bY_{s}) -b(\bY_{n\gamma})} } }\\
 & \leq &  2 \defEns{ \expe{\abs[2]{ \int_{n\gamma}^s \generatorD b(\bY_{n \gamma}, \bY_u)  \rmd u}} + \expe{\abs[2]{\int_{n\gamma}^s \langle \nabla b(\bY_u) ,\rmd B_u\rangle }}} \\
  & \leq &   2 \defEns{ \expe{ (s-n \gamma) \int_{n\gamma}^s \abs[2]{\generatorD b(\bY_{n \gamma}, \bY_u) } \rmd u} + \expe{\int_{n\gamma}^s\Tr\parenthese{ \nabla b \nabla b^{\transpose}}(\bY_u) \rmd u }}\eqsp. 
\end{eqnarray*}
By \Cref{bound_moment_bY}, for any $u \in \ccint{n \gamma,(n+1)\gamma}$,
\begin{equation}
\label{eq:proof_boundby_M_1_M_2_1}
  \expe{   \abs[2]{\generatorD b(\bY_{n \gamma}, \bY_u) }} \leq  \tM_1 \eqsp, \quad\text{and}\quad  \expe{\Tr\parenthese{ \nabla b \nabla b^{\transpose}}(\bY_u)  }  \leq \tM_2 \eqsp.
\end{equation}
Therefore, we get
\begin{equation}
   \expe{    \abs[2]{b(\bY_{s}) -b(\bY_{n\gamma})}}\leq   2 (s-n\gamma)^2 \tM_1 + 2 (s-n\gamma )\tM_2  \eqsp.
    \end{equation}
 We can now bound  $ \expe{\mathrm{B}_1} $. 
Let $t \in \ccint{n\gamma, (n+1)\gamma}$ and define
     \begin{equation}
\mathrm{B}_{11} =       {\expe{ \int_{n \gamma}^t \ps{Z_s} { \int_{n\gamma}^s\langle\nabla b(\bY_u), \rmd B_u \rangle} \rmd s }} \eqsp.
\end{equation}
By \eqref{eq:def_I_2}, \eqref{eq:ito_b_bY}, the Cauchy-Schwarz inequality and  \eqref{eq:proof_boundby_M_1_M_2_1}, 
    \begin{eqnarray}
\nonumber
    {\expe{\mathrm{B}_1}} &\leq & \frac 12\expe{\int_{n\gamma}^t \abs[2]{Z_s} \rmd s } +  \frac 12\expe{\int_{n\gamma}^t\abs[2]{  \int_{n\gamma} ^u  \generatorD b(\bY_{n \gamma}, \bY_u) \rmd u } \rmd s}  + \mathrm{B}_{11} \\
      \label{eq:bound_I_1_langevin}
      &  \leq &\frac 12 {\int_{n\gamma}^t  \expe{\abs[2]{Z_s}}  \rmd s + \frac 16 \gamma^3 \tM_1   } + \mathrm{B}_{11}\eqsp .
    \end{eqnarray}
        We now bound $\abs{\mathrm{B}_{11}}$. Denote by  $(\mcf_{\tilde{t}}^B)_{{\tilde{t}} \geq 0}$ the filtration associated with $(B_{\tilde{t}})_{\tilde{t} \geq 0}$.     Note that since $(\int_{0}^{\tilde{t}} \nabla b ( \bY_{s}) \rmd B_{s})_{{\tilde{t}} \geq 0}$ is a $(\mcf_{\tilde{t}}^B)_{{\tilde{t}} \geq 0}$-martingale and using that  $(Y_{\tilde{t}},\bY_{\tilde{t}})_{\tilde{t} \geq 0}$ is $(\mcf_{\tilde{t}}^B)_{\tilde{t} \geq 0}$-adapted, we have for any $t \in \ccint{n \gamma, (n+1)\gamma}$, and $u \in \ccint{0,n\gamma}$, $\expeLigne{\psLigne{b(Y_{u}) - \tbg(\bY_{u})}{\int_{n\gamma}^{{t}} \nabla b ( \bY_{{s}}) \rmd B_{s}}} = 0$ and $\expeLigne{\psLigne{Z_u}{\int_{n\gamma}^{{t}} \nabla b ( \bY_{{s}}) \rmd B_{s}}} = 0$. 
    Therefore, by  Fubini's theorem, \eqref{eq:expre_theta_langevin} and the Cauchy-Schwarz inequality, we obtain for any $t \in \ccint{n\gamma,(n+1)\gamma}$,
\begin{align}
  \nonumber
  \mathrm{B}_{11} &=    \int_{n \gamma}^t  {\expe{  \ps{Z_s-Z_{n\gamma} } { \int_{n\gamma}^s\langle \nabla b(\bY_u), \rmd B_u\rangle }   } } \rmd s\\
  \nonumber
         & = { \int_{n \gamma}^t \expe{\ps{\int_{n \gamma }^s \{ b(Y_u) - \tbg(\bY_{n \gamma} )  \} \rmd u } { \int_{n\gamma}^s\langle\nabla b(\bY_u), \rmd B_u\rangle }    } \rmd s }  \\
           \nonumber
         & = { \int_{n \gamma}^t \expe{\ps{\int_{n \gamma }^s \{ b(Y_u) - b(Y_{n \gamma} )  \} \rmd u } { \int_{n\gamma}^s\langle\nabla b(\bY_u), \rmd B_u\rangle }    } \rmd s }  \\
         & \leq \parentheseDeux{ \int_{n \gamma}^t \expeE{ \abs[2]{\int_{n \gamma }^s \{ b(Y_u) - b(Y_{n \gamma} )  \} \rmd u}  }\rmd s }^{1/2}
         \,  \left( \frac{\gamma^2}2\tM_2\right)^{1/2} ,
  \label{eq:first_bound_I_2_langevin}
\end{align}
where we have used in the last step that by It\=o's isometry,
$${ \int_{n \gamma}^t \expeE{ \abs[2]{ \int_{n\gamma}^s\langle\nabla b(\bY_u) ,\rmd B_u\rangle } } \rmd s }
=\int_{n \gamma}^t  \int_{n\gamma}^s \expeE{  \Tr(\nabla b \nabla b^{\transpose})(\bY_u) }\rmd u  \, \rmd s \le \frac{\gamma^2}2\tM_2 \eqsp. $$
Moreover, analogously as in \eqref{eq:lem:bound_w_2_euler_a}, \eqref{eq:lem:bound_w_2_euler_b} and \eqref{eq:lem:bound_w_2_euler_b_a}, we obtain 
\begin{equation}
 \label{eq:second_bound_I_2_langevin}
 \expeE{ \abs[2]{\int_{n \gamma }^s \{ b(Y_u) - b(Y_{n \gamma} )  \} \rmd u}  }\rmd s \ \le \frac{\gamma^4}{2}\tM_3+\gamma^3\tM_4 .
\end{equation}
The only difference to the argument used above is that now the law of $Y_u$
is $\pig \rmP_u$ instead of $\pi$, and therefore the constants $M_1$ and $M_2$ appearing in \eqref{eq:lem:bound_w_2_euler_b} and \eqref{eq:lem:bound_w_2_euler_b_a} are replaced by $\tM_3$ and $\tM_4$, respectively.\smallskip

By combining \eqref{eq:bound_I_1_langevin}, \eqref{eq:first_bound_I_2_langevin} and \eqref{eq:second_bound_I_2_langevin}, we conclude that
$$  {\expe{\mathrm{B}_1}} 
      \leq \ \frac 12 {\int_{n\gamma}^t \expe{ \abs[2]{Z_s}}  \rmd s + \frac 16 \gamma^3 \tM_1   } +\frac 12\gamma^{7/2}\tM_2^{1/2}\tM_3^{1/2} + \frac{1}{\sqrt 2}\gamma^3\tM_2^{1/2}\tM_4^{1/2}
      \eqsp .$$
Finally, by Cauchy-Schwarz, \Cref{as:tilde_b_b} and since $\bY_{n\gamma}$ has distribution $\pig$, 
\begin{equation}
{\expe{\mathrm{B}_2}} \leq \frac 12{\int_{n \gamma}^t\expe{ \abs[2]{Z_s}} \rmd s  + \frac{\gamma^2}2\int_{n\gamma}^t \expe{\bfGamma^2(\bY_{n\gamma})} \rmd s }  \leq \frac 12{\int_{n \gamma}^t \expe{\abs[2]{Z_s}} \rmd s  +\frac{ \gamma^3}2 \tM_5} . 
\end{equation}
Taking expectations in \eqref{eq:0:lem:bound_w_2_euler_altern} and inserting the bounds completes the proof.
\end{proof}

\begin{proof}[Proof of \Cref{theo:bias_final_ula_2}]
The result is a direct consequence of \Cref{theo:bound_w_2_euler_altern}
and the inequalities in \eqref{eq:accuracyinvariantmeasure2} and \eqref{eq:accuracyinvariantmeasure3}.
\end{proof}



\subsection{Proofs of \Cref{theo:regular_transf} and \Cref{theo:bound_tv}}
\label{sec:proofs-crefth-crefth_tv}

Define
for all $\gamma >0$, the function $\nFun : \ooint{0,\plusinfty} \to \nset$ by
\begin{equation}
\label{eq:def_nFun}
  \nFun(\gamma) = \ceil{\log\parenthese{\gamma^{-1}}/ \log(2)} \eqsp.
\end{equation}

\begin{proof}[Proof of \Cref{theo:regular_transf}]
Under \Cref{as:b_lip}-\Cref{ass:one_side_lip_b} and $\sup_{x \in\rset^d} \ps{b(x)}{x} < \plusinfty$,  \cite[Theorem 19]{debortoli2019convergence} shows that for any $t \geq 0$,
  \begin{equation}
    \label{eq:17}
          \tvnorm{\updelta_x \rmP_t - \updelta_y \rmP_t} \leq \limsup_{k \to \plusinfty} \tvnorm{\updelta_x \rmR_{t/k}^k  - \rmR_{t/k}^k} \eqsp,
        \end{equation}
        where $\rmR_{t/k}$ is given by \eqref{eq:def_rmR_gamma} with $\tb_{t/k} \equiv b$. 
        Note that by \Cref{as:b_lip} and \Cref{ass:one_side_lip_b}, for any $x,y \in\rset^d$ and $\gamma\in\ocint{0,\bgamma}$, $\norm{x+\gamma b(x) - \{y+\gamma b(y)\}}^2 \leq (1+\gamma \upkappa(\gamma)) \norm{x-y}^2$, with $\upkappa(\gamma) = 2\kappa + L^2\gamma$. Therefore, by \cite[Theorem 19]{durmus:moulines:bernoulli}, for any $\gamma \in\ocint{0,\bgamma}$ and $k \in\nset$,
        \begin{equation}
          \label{eq:18}
          \tvnorm{\updelta_x \rmR_\gamma ^k - \updelta_y \rmR_\gamma ^k} \leq 1 - 2 \Phibf\parenthese{ - \{\upkappa(\gamma)\}^{1/2}\frac{\norm{x-y}}{2 \sqrt{1-(1+\upkappa(\gamma)\gamma)^{k+1}}}} \eqsp,
        \end{equation}
        where $\Phibf$ is the cumulative distribution function of the standard one-dimensional Gaussian distribution. 
          Combining this result with \eqref{eq:17} completes the proof upon using that $1 - 2 \Phibf(-u) \leq u\sqrt{2/\uppi}$ for any $u \geq 0$. 
        \end{proof}

        \begin{proof}[Proof of \Cref{theo:bound_tv}]
Let $\gamma \in\ocint{0,\bgamma}$ and set $t_{\gamma} = \gamma 2^{\nFun(\gamma)}$ with $\nFun(\gamma)$ defined in \eqref{eq:def_nFun}.  We consider the following decomposition 
  \begin{equation}
    \label{eq:20}
        \tvnorm{\pi-\pi_{\gamma}} \leq         \tvnorm{\pi \rmP_{t_{\gamma}} -\pi_{\gamma}\rmP_{t_{\gamma}}} +\tvnorm{\pi_{\gamma} \rmP_{t_{\gamma}} - \pi_{\gamma} } \eqsp. 
      \end{equation}
      First by \Cref{ass:wasser_to_tv}-\ref{ass:wasser_to_tv_ii} and \Cref{lem:tv_to_wasser}, we get using $2^{\nFun(\gamma)} \gamma \geq 1$
      \begin{equation}
        \label{eq:21}
        \tvnorm{\pi \rmP_{t_{\gamma}} -\pi_{\gamma}\rmP_{t_{\gamma}}} \leq  \CTV \BTV \chi(1) \, \gamma \eqsp. 
      \end{equation}
      It remains to bound the second term in \eqref{eq:20} for which we apply \Cref{lem:bound_tv_lemma} and the bound
      \begin{equation}
        \label{eq:24}
        \sum_{k=1}^{\nFun(\gamma)} \CTV \ATV  \chi(2^{k-1} \gamma)[ \gamma^{3} 2^{k-1}]^{1/2}\exp(\lambdaTV 2^{k-1}\gamma) \leq  4\nFun(\gamma) \gamma  \CTV \ATV \bar{\chi} \exp(\lambdaTV 2)\eqsp,
      \end{equation}
   where we  have used $2^{\nFun(\gamma)} \gamma \leq 4$.
\end{proof}

\begin{lemma}
  \label{lem:tv_to_wasser}
  Assume \Cref{assum:diffusion_semigroup_a} and \Cref{ass:wasser_to_tv}-\ref{ass:wasser_to_tv_iii}.
  Then, for any probability measure $\mu,\nu \in \mathcal{P}_1(\rset^d)$ and $t> 0$, we get
  \begin{equation}
    \label{eq:22}
    \tvnorm{\mu \rmP_t -\nu \rmP_t} \leq \CTV \chi(t) \wasserstein[1]{\mu,\nu} \eqsp.
  \end{equation}
\end{lemma}

\begin{proof}
  Let   $\mu,\nu \in \mathcal{P}_1(\rset^d)$ and $t > 0$.  First, for any coupling $\xi$ of $\mu$ and $\nu$, we easily get using $\tvnorm{\mu'-\nu'} = (1/2)\sup\{\absLigne{\int f \rmd \mu' - \int f \rmd \nu'} \, :\, \abs{f} \leq 1\}$, 
  \begin{equation}
    \label{eq:23}
   \tvnorm{\mu \rmP_t -\nu \rmP_t} \leq \int_{\rset^d \times \rset^d} \xi(\rmd x \rmd y ) \tvnorm{\updelta_x \rmP_t - \updelta_y \rmP_t} \eqsp. 
  \end{equation}
  Using \Cref{ass:wasser_to_tv}-\ref{ass:wasser_to_tv_iii} and taking for $\xi$ the optimal coupling between $\mu$ and $\nu$ for $\wassersteinDLigne[1]$ complete the proof.
\end{proof}

\begin{lemma}
  \label{lem:bound_tv_lemma}
      Assume \Cref{assum:diffusion_semigroup_a}, \Cref{as:r_gamma}, \Cref{as:b_lip} and \Cref{ass:wasser_to_tv}. Suppose in addition that $\bgamma  <1$.
Then for any $\gamma \in\ocint{0,\bgamma}$, $\ell \in \nset$, $\ell \geq 2^{\nFun(\gamma)}$,
\begin{multline}
  \tvnorm{\pi_{\gamma} \rmP_{\ell \gamma} - \pi_{\gamma}} \leq \CTV \ATV  \chi(2^n \gamma) [\gamma^{3} (\ell-2^{\nFun(\gamma)})]^{1/2}\exp(\lambdaTV (\ell-2^{\nFun(\gamma)})\gamma)\\
  + 2^{-3/2} L  \defEns{\gamma^3 \tilde{M}_6/3 +d \gamma^2}^{1/2} 
+ \sum_{k=1}^{\nFun(\gamma)} \CTV \ATV  \chi(2^{k-1} \gamma)[ \gamma^{3} 2^{k-1}]^{1/2}\exp(\lambdaTV 2^{k-1}\gamma)   \eqsp,
\end{multline}
where $\tilde{M}_6$ is defined in \eqref{eq:def_tilde_M_6}.
  \end{lemma}
\begin{proof}

  For ease of notation, denote $n = \nFun(\gamma)$ for $ \gamma  \in\ocint{0,\bgamma}$. Let $\ell \in \nset$ such that $\ell  \geq 2^n$.
    Consider the following decomposition
\begin{align}
\nonumber
&  \tvnorm{\pi_{\gamma} \rmP_{\ell \gamma} - \pi_{\gamma}} = \tvnormEq{\pi_{\gamma} \rmP_{\ell \gamma} - \pi_{\gamma} \rmR_{\gamma}^\ell} \leq
\tvnormEq{\defEns{\pi_{\gamma} \rmP_{(\ell-2^n) \gamma} - \pi_{\gamma} \rmR_{\gamma}^{\ell-2^n}} \rmP_{2^n \gamma}} \\
\label{eq:decomp_proof_tv_fix_1_2}
& \qquad \qquad + \tvnormEq{\pi_{\gamma}\rmR_{\gamma}^{\ell-1}\defEns{\rmP_{\gamma}-\rmR_{\gamma}}}
+ \sum_{k=1}^{n} \tvnormEq{\pi_{\gamma} \rmR_{\gamma}^{\ell-2^{k}}\defEns{  \rmP_{2^{k-1}\gamma}-\rmR_{\gamma}^{2^{k-1}}} \rmP_{2^{k-1} \gamma} }   \eqsp.
\end{align}
We bound each term in the right hand side.
First by \Cref{lem:tv_to_wasser} and \Cref{ass:wasser_to_tv}-\ref{ass:wasser_to_tv_i}, we have
\begin{align}
  \nonumber
  & \tvnormEq{\defEns{\pi_{\gamma} \rmP_{(\ell-2^n) \gamma} - \pi_{\gamma} \rmR_{\gamma}^{\ell-2^n}} \rmP_{2^n \gamma}} \\
    \label{eq:first_bound_proof_tv_fix_2}
& \qquad \qquad  \leq \CTV \ATV  \chi(2^n \gamma) [\gamma^{3} (\ell-2^n)]^{1/2}\exp(\lambdaTV (\ell-2^n)\gamma)  \eqsp.
\end{align}
Similarly we get for all $k \in \defEns{1,\cdots,2^n}$,
\begin{align}
  \nonumber
&\tvnormEq{ \pi_{\gamma} \rmR_{\gamma}^{\ell-2^{k}}\defEns{  \rmP_{2^{k-1}\gamma}-\rmR_{\gamma}^{2^{k-1}}} \rmP_{2^{k-1} \gamma}} = \tvnormEq{ \pi_{\gamma} \defEns{  \rmP_{2^{k-1}\gamma}-\rmR_{\gamma}^{2^{k-1}}} \rmP_{2^{k-1} \gamma}}
  \\
  \label{eq:second_bound_proof_tv_fix_2}
  & \qquad \qquad \leq  \CTV \ATV  \chi(2^{k-1} \gamma) [\gamma^{3/2} 2^{k-1}]^{1/2}\exp(\lambdaTV 2^{k-1}\gamma)  \eqsp.
\end{align}
For the last term, adapting the proof of  \cite[Proposition 2]{durmus:moulines:2016} to a general drift $b$ in place of $\nabla U$,  we have
\begin{equation}
 \tvnormEq{\pi_{\gamma} \rmR_{\gamma}^{\ell-1}\defEns{\rmP_{\gamma}-\rmR_{\gamma}}}^2  \leq
2^{-3} L^2  \defEns{\gamma^3 \tilde{M}_6/3 +d \gamma^2} \eqsp.
\end{equation}
Combining this inequality, \eqref{eq:first_bound_proof_tv_fix_2}, \eqref{eq:second_bound_proof_tv_fix_2}  in \eqref{eq:decomp_proof_tv_fix_1_2} concludes the proof.
\end{proof}

\subsection{Postponed proofs of  \Cref{subsec:HMC}}
\label{sec:proof-crefth}

\begin{lemma}
  \label{lem:trapeze_method_bound}
  Let $f \in \rmC^{2}(\ccint{a,b}, \rset^d)$ for $a,b \in \rset$, $a <b$. Then for any $t \in \ccint{a,b}$,
  \begin{equation}
      \label{eq:trapeze_method_bound_2}
   \int_{a}^{b} \{ f(t) - (f(b)+f(a))/2\} \rmd t =  - (1/2) \int_{a}^b f''(t)(b-t)(t-a)\rmd t\eqsp. 
 \end{equation}
\end{lemma}
\begin{proof}
    For any $t \in \ccint{a,b}$, since $f \in \rmC^{2}(\ccint{a,b}, \rset^d)$, we have using integration by parts twice
  \begin{align}
&    f(t)-\frac{(b-t)f(a) + (t-a)f(b)}{b-a}  = \parentheseDeux{ \frac{(b-t)}{b-a}\int_a^t f'(s) \rmd s -\frac{(t-a)}{b-a}  \int_t^b f'(s) \rmd s } \\
& \qquad  = \frac{(b-t)}{b-a}\defEns{ (t-a) f'(t) - \int_a^t f''(s)(s-a) \rmd s }  \\
    &\qquad\qquad +\frac{(t-a)}{b-a}\defEns{ (t-b) f'(t) - \int_t^b f''(s)(s-b) \rmd s }   \eqsp,
  \end{align}
  which implies
    \begin{multline}
  \label{eq:trapeze_method_bound_1}
(b-a)    f(t)  = (b-t)f(a) + (t-a)f(b) +  \frac{(b-t)}{b-a}\int_{a}^s f''(s)(s-a) \rmd s\\ - \frac{(t-a)}{b-a} \int_{t}^b f''(s) (b-s) \rmd s \eqsp,
  \end{multline}
  Now integrating  this identity over $\ccint{a,b}$, we obtain
  \begin{align}
   \int_{a}^{b}  f(t)   \rmd t    
    & = (f(b)+f(a))/2 +   \int_a^b \frac{(s-a) f''(s)}{b-a} \int_{s}^b (b-t) \rmd t  \rmd s \\
    &     \quad \quad  - \int_a^b \frac{(b-s) f''(s)}{b-a} \int_{a}^u (a-t) \rmd t  \rmd s \eqsp,
  \end{align}
  which implies \eqref{eq:trapeze_method_bound_2}. 
\end{proof}

\begin{lemma}
\label{lem:asympto_ihmc}
  Assume \Cref{assum:exact_hmc_kernel}, \Cref{ass:app_F_h} and \Cref{as:F_lip} and let $\gamma \in\ocint{0,\bgamma}$.
  Let $\tilde{\pi} \in \Pens_2(\rset^d)$, and let $G,Q$ be 
  $\mathbb R^d$-valued random variables such that  $G$ is normally distributed with zero-mean and covariance matrix identity, $Q$ has distribution $\pi$ and is independent of $G$.  Define $(\rmX_k,\rmV_k)_{k \ge 0}$ and $(\tilde{\rmX}_k,\tilde{\rmV}_k)_{k \ge 0}$ recursively by $\rmX_0=\tilde{\rmX}_0= Q$, $\rmV_0 = \tilde{\rmV}_0 = G$, and for any $k \in \nset$, 
  $$(\rmX_{k+1},\rmV_{k+1})= \uppsi_{\gamma}(\rmX_k,\rmV_k), \quad (\tilde{\rmX}_{k+1},\tilde{\rmV}_{k+1})= \Psiverlet_{\gamma}(\tilde{\rmX}_k,\tilde{\rmV}_k) ,$$ 
where $\uppsi_\gamma$ and $\Psiverlet_{\gamma}$ are defined by  \eqref{eq:def_hamil_ode}  and \eqref{eq:defPhiverlt_h}. Then for any integer $k \ge 0$,
  \begin{align}
\nonumber  
&   \expeE{\abs[2]{    \rmX_{k+1}-\tilde{\rmX}_{k+1}}}[1/2]\ \le \gamma^3( (M_2/12)^{1/2}+ M_5^{1/2}/2)\\
\label{eq:1:lem:asympto_ihmc}&  \qquad \qquad  +\left(1+\frac{\gamma^2\Ltt}{2}\right)\expeE{\abs[2]{\rmX_{k}-\tilde{\rmX}_{k}}}[1/2] + {\gamma}\,\expeE{\abs[2]{   \rmV_k -\trmV_k}}[1/2]
 \eqsp,\\
 \nonumber
\lefteqn{    \expeE{\abs[2]{    \rmV_{k+1} - \tilde{\rmV}_{k+1}}}[1/2] 
\ \le \   \frac{{\gamma}^3}{2}\left( 
 \left(2+\frac{\gamma\Ltt}{2} \right)M_5^{1/2} 
    +   \frac{\gamma \Ltt M_2^{1/2}}{2 \sqrt{12}}+ \frac{(M_1+2M_4)^{1/2}}{15^{1/2}}\right)
}\\
  \label{eq:2:lem:asympto_ihmc}  & \qquad \qquad + 
     \left(\gamma\Ltt+\frac{{\gamma}^3\Ltt^2}4\right) \expeE{\abs[2]{\rmX_{k}- \trmX_{k}}}[1/2]+
   \left(1+\frac{\gamma^2\Ltt}{2}\right)  \expeE{\abs[2]{    \rmV_{k} - \tilde{\rmV}_{k}}}[1/2] 
  \end{align}
  where  $M_1,M_2$, $M_4$ and $M_5$ are given by \eqref{eq:def_m_1_m_2_m_4} and \eqref{eq:def_m_5}.
\end{lemma}

\begin{proof}
 For any $s \ge 0$ let $(\bfrmX_s,\bfrmV_s) =  \uppsi_s(Q,G)$. Note that since $(\uppsi_s)_{s \in \rset_+}$ is the flow associated with \eqref{eq:def_hamil_ode}, we have by definition that for any $k \in \nset$, $(\rmX_k,\rmV_k) = (\bfrmX_{k {\gamma}} , \bfrmV_{k{\gamma}})$, and for any $t,s \geq 0$ with $ s \leq t$,
  \begin{equation}
    \bfrmX_t = \bfrmX_s + \int_s^t \bfrmV_u \rmd u  \eqsp, \qquad     \bfrmV_t = \bfrmV_s + \int_s^t b(\bfrmX_u) \rmd u \eqsp. 
  \end{equation}
    Therefore for any $k \in \nset$, using \eqref{eq:defPhiverlt_h}, we  have that
  \begin{align}
    \rmX_{k+1}-\tilde{\rmX}_{k+1} &= \rmX_k + \int_{k\gamma} ^{(k+1)\gamma} \bfrmY_s \rmd s - \trmX_{k} - \gamma^2 \tbg (\trmX_k) /2 - \gamma \trmY_k   \\
    \label{eq:hmc_inexact_X_1}
                                  &= \rmX_k-\trmX_{k} + \int_{k\gamma} ^{(k+1)\gamma} \int_{k\gamma}^s \{b(\bfrmX_u) -  \tbg (\trmX_k) \} \rmd u + \gamma(\rmY_k -\trmY_k) \eqsp,  \\
        \label{eq:hmc_inexact_Y_1}
    \rmY_{k+1} - \tilde{\rmY}_{k+1} & =      \rmY_{k} - \tilde{\rmY}_{k} + \int_{k\gamma}^{(k+1)\gamma} b(\bfrmX_s) \rmd s - \gamma (\tbg(\trmX_{k+1}) +\tbg(\trmX_k))/2 \eqsp. 
  \end{align}
  In addition, since $(\rmX_0,\rmY_0)$ has distribution $\pi\otimes \mu_{0,\Idd}$, then by  \Cref{assum:exact_hmc_kernel}, for any $s \geq 0$ and $k \in\nset$, $(\bfrmX_s,\bfrmV_s)$ and $(\rmX_k,\rmY_k)$ have distribution  $\pi\otimes \mu_{0,\Idd}$.
We first establish \eqref{eq:1:lem:asympto_ihmc}. By \eqref{eq:hmc_inexact_X_1}, the Minkowski inequality, \Cref{ass:app_F_h} and since $\trmX_k = \Phiverlet[h][k](Q,Z)$,  we have that
  \begin{align}
    & \expeE{\abs[2]{    \rmX_{k+1}-\tilde{\rmX}_{k+1}}}[1/2] \leq \expeE{\abs[2]{\rmX_{k}-\tilde{\rmX}_{k}}}[1/2] + \gamma\expeE{\abs[2]{   \rmY_k -\trmY_k}}[1/2]  \\
    &   + (\gamma^2/2) \, \expeE{\abs[2]{\tbg(\trmX_k) -  b (\trmX_k)  }}[1/2] + \expeE{\abs[2]{  \int_{k\gamma} ^{(k+1)\gamma} \int_{k\gamma}^s \{b(\bfrmX_u) -  b (\trmX_k) \} \rmd u }}[1/2] 
        \label{eq:hmc_inexact_X_2}
  \end{align}
  Now using the  Minkowski and  Cauchy-Schwarz inequalities, we obtain that
  \begin{align}
    &     \expeE{\abs[2]{  \int_{k\gamma} ^{(k+1)\gamma} \int_{k\gamma}^s \{b(\bfrmX_u) -  b (\trmX_k) \} \rmd u \rmd s  }}[1/2]   \leq (\gamma^2/2) \expeE{\abs[2]{   b(\rmX_k) -  b (\trmX_k)   }}[1/2] \\
    &   \qquad\qquad\qquad + \gamma^{1/2 } \expeE{  \int_{k\gamma} ^{(k+1)\gamma} (s-k\gamma) \int_{k\gamma}^s \abs[2]{b(\bfrmX_u) -  b (\bfrmX_{k\gamma}) } \rmd u \rmd s  }[1/2] \\
    & \leq  (\gamma^2/2) \expeE{\abs[2]{   b(\rmX_k) -  b (\trmX_k)   }}[1/2]\\
    &\qquad +  \gamma^{1/2} \expeE{  \int_{k\gamma} ^{(k+1)\gamma} (s-k\gamma) \int_{k\gamma}^s (u-k\gamma) \int_{k\gamma}^u  \abs[2]{\generatorH b (\bfrmX_v , \bfrmY_v) } \rmd v  \rmd u \rmd s  }[1/2] \eqsp,
  \end{align}
  where we used the Cauchy-Schwarz inequality again for the last upper bound.
  We obtain by \Cref{as:F_lip} and \eqref{eq:Liouvillenormsquared} that
  \begin{multline}
    \expeE{\abs[2]{  \int_{k\gamma} ^{(k+1)\gamma} \int_{k\gamma}^s \{b(\bfrmX_u) -  b (\trmX_k) \} \rmd u \rmd s  }}[1/2]  \\\leq
    (\gamma^2\Ltt/2) \expeE{\abs[2]{   \rmX_k -  \trmX_k   }}[1/2]  + \gamma^3(M_2/12)^{1/2} \eqsp.
  \end{multline}
  Plugging this result in \eqref{eq:hmc_inexact_X_2} and using \Cref{ass:app_F_h} and the definition of $M_5$ \eqref{eq:def_m_5}, we obtain \eqref{eq:1:lem:asympto_ihmc}.

    We now turn to showing \eqref{eq:2:lem:asympto_ihmc}. 
By \eqref{eq:hmc_inexact_Y_1}, the Minkowski and Cauchy-Schwarz inequalities, we have
  \begin{equation}
\label{eq:hmc_inexact_Y_2}
    \expeE{\abs[2]{    \rmY_{k+1} - \tilde{\rmY}_{k+1}}}[1/2] =   \expeE{\abs[2]{    \rmY_{k} - \tilde{\rmY}_{k}}}[1/2] +(\gamma/2)A + B \eqsp,
  \end{equation}
  where
  \begin{align}
    A & =  \, \expeE{\abs[2]{ (b(\rmX_{k+1}) +b(\rmX_k)) - (\tbg(\trmX_{k+1}) +\tbg(\trmX_k))}}[1/2] \eqsp, \\
    B &=  \, \expeE{\abs[2]{  \int_{k\gamma}^{(k+1)\gamma} b(\bfrmX_s) \rmd s - \gamma (b(\rmX_{k+1}) +b(\rmX_k))/2 }}[1/2] \eqsp. 
  \end{align}
  We bound $A$ and $B$ separately. By the  Minkowski inequality, \Cref{ass:app_F_h} and \Cref{as:F_lip}, we have
  \begin{align}
    A 
    & \leq  \expeE{\abs[2]{ (b(\trmX_{k+1}) +b(\trmX_k)) - (\tbg(\trmX_{k+1}) +\tbg(\trmX_k))}}[1/2]\\
    &  \qquad \qquad  \qquad \qquad + \expeE{\abs[2]{ (b(\rmX_{k+1}) +b(\rmX_k)) - (b(\trmX_{k+1}) +b(\trmX_k))}}[1/2] \\
    & \leq 2 \gamma^2 M_5^{1/2} + \Ltt \defEns{ \expeE{\abs[2]{\rmX_{k+1}- \trmX_{k+1}}}[1/2]+  \expeE{\abs[2]{\rmX_{k}- \trmX_{k}}}[1/2] } \eqsp. 
  \end{align}
  Then using \eqref{eq:1:lem:asympto_ihmc}, we get that
  \begin{align}
  \label{eq:hmc_inexact_Y_3}
  &A \leq 2  \gamma^2 M_5^{1/2} +  \Ltt \defEns{\gamma  \expeE{\abs[2]{\rmY_{k}- \trmY_{k}}}[1/2]  +       (2+\gamma^2\Ltt/2) \expeE{\abs[2]{\rmX_{k}- \trmX_{k}}}[1/2] }
\\&\qquad \qquad   + \Ltt \defEns{ \gamma^3 M_5^{1/2}/2 + \gamma^3(M_2/12)^{1/2} } \eqsp. 
  \end{align}
  
  Using \eqref{eq:use_generatorH}, \Cref{lem:trapeze_method_bound} and the Cauchy-Schwarz inequality, we obtain that 
  \begin{align}
     &  2 B  =     \expeE{\abs[2]{\int_{k\gamma}^{(k+1)\gamma}  (\generatorH)^{2} b(\bfrmX_s,\bfrmY_s) \{ (s-k\gamma)((k+1)\gamma-s)\}\rmd s}} [1/2]   \\
     & \leq  \gamma^{1/2}  \expeE{\abs[2]{\int_{k\gamma}^{(k+1)\gamma} \abs{ (\generatorH)^{2} b(\bfrmX_s,\bfrmY_s)}^2 \{ (s-k\gamma)((k+1)\gamma-s)\}^2\rmd s}} [1/2] \\
    \label{eq:hmc_inexact_Y_4}
  & = \gamma^315^{-1/2} \{M_1 + 2 M_4\}^{1/2}
\end{align}
where we used for the last equality that  for any $s \geq 0$, 
$(\bfrmX_s , \bfrmV_s)$ has distribution $\pi\otimes \mu_{0,\Idd}$ and by \eqref{eq:Liouvillesquaredsquared},
$$\mathbb E\left[ \left|\left((\generatorH)^2 b\right)(\bfrmX_s , \bfrmV_s)\right|^2\right]= 
\int\int \left| \left((\generatorH)^2 b\right)(q,p)\right|^2    \varphibf_{d}(p)\, \rmd p\, 
\pi (\rmd q) = 
M_1+2M_4\eqsp .
$$
Combining \eqref{eq:hmc_inexact_Y_3}-\eqref{eq:hmc_inexact_Y_4} in \eqref{eq:hmc_inexact_Y_2} concludes the proof of \eqref{eq:2:lem:asympto_ihmc}.

\end{proof}

\begin{proof}[Proof of \Cref{theo:wasserstein_2_bound}]
  Let $n=T/\gamma$, and let $(\rmX_k,\rmV_k)_{k \ge 0}$ and $(\tilde{\rmX}_k,\tilde{\rmV}_k)_{k \ge 0}$ be defined as in \Cref{lem:asympto_ihmc}. Then by definition of the transition kernels, $\rmX_n$ and $\tilde{\rmX}_n$ have law $ \pi  \rmK_{T}$ and 
$\pi  \rmK_{T,\gamma}$, respectively, and thus
    \begin{equation}
      \label{eq:wasser_hmc_theo_1}
          \wasserstein[2]{ \pi  \rmK_{T}, \pi  \rmK_{T,\gamma}}\leq       \expeE{\abs[2]{    \rmX_{n}-\tilde{\rmX}_{n}}}[1/2] .
    \end{equation}
Now
    consider the sequence in $\rset^2$ defined by
    \begin{equation}
     \label{eq:55}
      z_k  =  \parenthese{ \expeELigne{\absLigne[2]{    \rmX_{k}-\tilde{\rmX}_{k}}}[1/2]\, ,\, \Ltt^{-1/2} \expeELigne{\absLigne[2]{    \rmV_{k}-\tilde{\rmV}_{k}}}[1/2]}^{\transpose} ,\quad k \ge 0\eqsp.
    \end{equation}
    By \Cref{lem:asympto_ihmc}, for any $k \ge 0$, we have
    \begin{equation}
      \label{eq:56}
      z_{k+1}\leq_2\ \rmA z_{k} +\frac{1}{12}\gamma^3 \vec{M} \eqsp,
    \end{equation}
    where $\leq_2$ is the partial order on $\rset^2$ defined by $(u,v)^{\transpose}\leq_2(\tilde u,\tilde v)^{\transpose}$
    if and only if $u\leq \tilde u$ and $v \leq \tilde v$,  
    \begin{align}\nonumber
      \rmA &=
      \begin{pmatrix}
        1+\gamma^2 \Ltt/2 & \gamma L^{1/2} +\gamma^3 \Ltt^{3/2}/4 \\
     \gamma \Ltt^{1/2} +\gamma^3 \Ltt^{3/2}/4& 1+\gamma^2 \Ltt/2 
      \end{pmatrix}
\eqsp \eqsp,\\
 \vec{M} &= \begin{pmatrix}
 (M_2/12)^{1/2}+ M_5^{1/2}/2\\
\frac{1}{2\Ltt^{1/2}}\left( 
 \left(2+\frac{\gamma\Ltt}{2} \right)M_5^{1/2} 
    +   \frac{\gamma \Ltt M_2^{1/2}}{2 \sqrt{12}}+ \frac{(M_1+2M_4)^{1/2}}{15^{1/2}}\right) \end{pmatrix} \eqsp .\label{eq:vecM}
\end{align}
Since $\rmA$ has positive entries, application of $\rmA$ preserves the partial order on $\rset^2$. Noting that $z_0=0$, a straightforward induction 
based on \eqref{eq:56} shows that for any integer $k\ge 0$,
\begin{equation}\label{eq:boundzn}
z_k\leq_2 \gamma^3\sum_{i=0}^{k-1}\rmA^i  \vec{M} . 
\end{equation}
Since $\rmA$ is symmetric with maximal eigenvalue 
\begin{equation}\label{eq:maxev}
\lambda_{\rmA}  = 1+ \gamma \Ltt^{1/2}+\gamma^2 \Ltt/2  +\gamma^3 \Ltt^{3/2}/4 \eqsp,
\end{equation}
by \eqref{eq:boundzn}-\eqref{eq:55}  and the triangle inequality, we have 
\begin{equation}\label{eq:boundznnorm}
\expeE{\abs[2]{    \rmX_{n}-\tilde{\rmX}_{n}}}[1/2]\ \le 
\left|z_n\right| \leq \frac{1}{12}\gamma^3\sum_{k=0}^{n-1}\lambda_{\rmA}^k  |\vec{M}| =  \gamma^3\frac{\lambda_{\rmA}^n-1}{\lambda_{\rmA} -1}|\vec{M}|
\eqsp . 
\end{equation}
The assertion follows from this bound and \eqref{eq:wasser_hmc_theo_1}, because by \eqref{eq:maxev}, $\lambda -1\ge \gamma \Ltt^{1/2}$ and
$\lambda^n\le \exp\left( n \gamma \lambda_{\Hamiltonian}\right)$, and by \eqref{eq:vecM}, $|\vec{M}|^2\le \Ltt^{-1}M_{\Hamiltonian}$.
%
%
\end{proof}


\bibliographystyle{amsplain}
\bibliography{bibliography}





\end{document}